\documentclass[11pt,english]{amsart}
\usepackage[english,frenchb]{babel}
\usepackage[latin1]{inputenc}
\usepackage[T1]{fontenc}
\usepackage{lmodern}
\usepackage{amsfonts}
\usepackage{amsmath}
\usepackage{amssymb}
\usepackage{url}
\usepackage{hyperref}
\usepackage{fancyhdr}
\newtheorem{theorem}{Th\'eor\`eme}[section]
\newtheorem{lemma}[theorem]{Lemme}

\newtheorem{corollary}[theorem]{Corollaire}

\theoremstyle{definition}
\newtheorem{definition}[theorem]{D\'efinition}

\newtheorem{remark}[theorem]{Remarque}

\numberwithin{equation}{section}

\def\Q{\mathbb{Q}}
\def\Z{\mathbb{Z}}

\let\sst=\scriptscriptstyle

\let\ds=\displaystyle

\def\order{\raise1.5pt \hbox{${\scriptscriptstyle \#}$}}
\def\Frac#1#2{\hbox{\footnotesize $\ds \frac{#1}{#2}$}}
\def \go{\leavevmode
\raise.3ex\hbox{$\sst\langle\!\langle$}%
~\ignorespaces\!}
\def \gf{\!\relax \ifhmode \unskip~\else \leavevmode \fi
\raise.3ex\hbox{$\sst\rangle\!\rangle$}\ }
\def\lien{\mathrel{\mkern-4mu}}
\def\too{\relbar\lien\rightarrow}
\def\tooo{\relbar\lien\relbar\lien\too}
\def\toooo{\relbar\lien\relbar\lien\tooo}

\def\Cl{{\mathcal C}\hskip-2pt{\ell}}
\def\cl{c\hskip-1pt{\ell}}
\def\plus{\ds\mathop{\raise 2.0pt \hbox{$\bigoplus$}}\limits}
\def\prd{\ds\mathop{\raise 2.0pt \hbox{$\prod$}}\limits}
\def\sm{\ds\mathop{\raise 2.0pt \hbox{$\sum$}}\limits}
\title[Conjecture de Greenberg]
{Normes d'id\'eaux  dans la tour cyclotomique \\ et conjecture de Greenberg \\ 
{\footnotesize (hypoth\`eses $p$-adiques sur les normes d'id\'eaux)} }
\author{Georges  Gras}
\address{Villa la Gardette, chemin Ch\^ateau Gagni\`ere,
 38520 Le Bourg d'Oisans \\ 
{\rm\url{http://www.researchgate.net/profile/Georges_Gras}}}
\email{g.mn.gras@wanadoo.fr}

\begin{document}


\maketitle

\begin{abstract}
Pre-print d'un article publ\'e dans ``Annales math\'ematiques du Qu\'ebec''.
Soit $k$ un corps de nombres totalement r\'eel  et soit
$k_\infty$ sa $\Z_p$-extension cyclotomique.
Ce travail prolonge notre article \go Approche $p$-adique 
de la conjecture de Greenberg pour les corps totalement r\'eels\gf 
au moyen d'heuristiques sur le comportement $p$-adique de normes
d'id\'eaux dans $k_\infty/k$~; en effet, cette conjecture 
(sur la nullit\'e des invariants 
$\lambda$ et $\mu$ d'Iwasawa) d\'epend d'images de ces normes
dans le groupe de torsion ${\mathcal T}_k$ du groupe de Galois de 
la pro-$p$-extension ab\'elienne $p$-ramifi\'ee maximale de $k$,
donc de leurs symboles d'Artin dans une extension 
finie $F/k$ obtenue par descente galoisienne de ${\mathcal T}_k$.
Une hypoth\`ese naturelle de r\'epartition de ces symboles 
implique $\lambda=\mu=0$. Des statistiques 
dans le cas quadratique confirment la 
probable exactitude de telles propri\'et\'es qui constituent 
l'obstruction fondamentale \`a une preuve de la conjecture de 
Greenberg dans le seul cadre de la th\'eorie d'Iwasawa.

\noindent
{\bf Abstract}
Pre-print of a publication in `` Annales math\'ematiques du Qu\'ebec ''.
Let $k$ be a totally real number field and let $k_\infty$ be its cyclotomic 
$\Z_p$-extension. This work continues our article \go Approche 
$p$-adique de la conjecture de Greenberg pour les corps totalement 
r\'eels\gf by means of heuristics on the $p$-adic behavior of ideal 
norms in $k_\infty/k$; indeed, this conjecture 
(on the nullity of the Iwasawa 
invariants $\lambda$, $\mu$) depends on some images of these 
norms in the torsion group ${\mathcal T}_k$ of the Galois group 
of the maximal abelian $p$-ramified pro-$p$-extension of $k$, 
thus of their Artin symbols in a finite extension $F/k$ obtained by 
Galois descent of ${\mathcal T}_k$. A natural assumption of 
distribution of these Artin symbols implies $\lambda=\mu=0$. 
Statistics in the quadratic case confirm the probable 
exactness of such properties
which constitute the fundamental obstruction for a proof 
of Greenberg's conjecture in the sole framework 
of Iwasawa's theory.

\noindent
{\bf Mots-cl\'es}{Greenberg's conjecture, Iwasawa's theory, 
$p$-class groups, class field theory, $p$-adic regulators, 
Fermat quotients of algebraic numbers}

\end{abstract}

\tableofcontents

\section{Introduction -- Contexte \go Conjecture de Greenberg\,\gf}\label{sec1}
Dans \cite{Gra3}, nous avons pos\'e une hypoth\`ese de 
r\'epartition \go uniforme\gf de symboles d'Artin convenables
de normes ${\rm N}_{k_n/k}({\mathfrak A})$
d'id\'eaux ${\mathfrak A}$ des \'etages $ k_n$ de la tour cyclo\-tomique 
$k_\infty = \bigcup_{n \geq 0} k_n$ d'un corps de nombres totalement r\'eel $k$
relativement \`a un nombre premier donn\'e $p>2$ totalement d\'ecompos\'e.
En effet, on constate qu'il existe des obstructions $p$-adiques \`a une 
preuve, dans le seul cadre de la th\'eorie d'Iwasawa alg\'ebrique, de 
la conjecture de Greenberg sur la nullit\'e des invariants $\lambda$ et $\mu$
pour les corps {\it totalement r\'eels} \cite[Theorems 1 and 2]{Gre1}, 
\cite[Conjecture 3.4]{Gre2}, \'etant entendu que cette
conjecture se pose quelle que soit la d\'ecomposition de $p$
d'apr\`es le point de vue de Jaulent~\cite{J2}. Cette hypoth\`ese de r\'epartition
met en jeu le groupe de torsion ${\mathcal T}_k$ (fini) de la pro-$p$-extension
ab\'elienne $p$-ramifi\'ee maximale $H_k^{\rm pr}$ de $k$ par l'interm\'ediaire 
de sa descente galoisienne comme groupe de Galois d'une extension 
ab\'elienne finie $F/k$, pouvant \^etre explicit\'ee.
Cette hypoth\`ese peut se r\'esumer par le fait que dans l'algorithme 
de \go d\'evissage\gf du $p$-groupe des classes de $k_n$,
les id\'eaux ${\mathfrak A}$ repr\'esentant ces classes sont tels que les
symboles d'Artin
$\Big(\Frac{F/k}{{\rm N}_{k_n/k}({\mathfrak A})}\Big)$ 
engendrent ${\rm Gal}(F/k)$ en un nombre d'\'etapes (de l'algorithme)
ind\'ependant de $n\gg 0$. Sous ces conditions 
(cf. Hypoth\`eses (H) \'enonc\'ees \`a la fin du \S\,\ref{sub2}), 
la conjecture de Greenberg en r\'esulte.

\smallskip
Pour un historique sur la conjecture de Greenberg, et sur les travaux 
pr\'ecurseurs de Ozaki et Taya \cite{OT}, \cite{Oz}, \cite{Ta}, 
se reporter \`a \cite{Gra3} et \`a sa biblio\-graphie, ainsi qu'aux 
r\'ecents points de vue de Jaulent \cite{J2} et Nguyen Quang Do \cite{Ng}.

\section{Pro-$p$-extension ab\'elienne $p$-ramifi\'ee 
maximale -- Le groupe ${\mathcal T}_k$}\label{sub1}

Soit $k$ un corps de nombres galoisien r\'eel, de degr\'e $d$, et
de groupe de Galois $g$. 
Soit $S_k := \{{\mathfrak p} \mid p\}$ l'ensemble des $p$-places de $k$.
Sous la conjecture de Leopoldt pour $p$ dans $k_\infty$, on obtient
le sch\'ema ci-apr\`es (dit de la {\it $p$-ramification ab\'elienne}). Dans
toute la suite, on suppose $p>2$ totalement d\'ecompos\'e dans~$k$.

\smallskip
On d\'esigne par $\Cl_k$ le $p$-groupe des classes de $k$ 
et par $E_k$ le groupe des unit\'es globales
$p$-principales $\varepsilon \equiv 1 \pmod p$ de $k$.
Soit $U_k:=\bigoplus_{{\mathfrak p} \in S_k} U_{\mathfrak p}^1$ le 
$\Z_p$-module (de $\Z_p$-rang $d$) des unit\'es locales $p$-principales
o\`u chaque $U_{\mathfrak p}^1$ est le groupe des unit\'es 
$\overline {\mathfrak p}$-principales de la compl\'etion
 $k_{\mathfrak p}$ de $k$ en ${\mathfrak p} \in S_k$,
et $\overline  {\mathfrak p}$ l'ideal maximal pour $k_{\mathfrak p}$. 
Par hypoth\`ese, $U_k \simeq (\Z_p^\times)^d = (1+p\,\Z_p)^d$.

\smallskip
Soit $K := k_n$, de degr\'e $p^n$ sur $k$, le $n$-i\`eme \'etage
de la $\Z_p$-extension cyclotomique $k_\infty$ de $k$~; comme
c'est une extension galoisienne {\it r\'eelle} de $\Q$, sous la conjecture 
de Leopoldt le compos\'e de ses $\Z_p$-extensions 
(corps des fixes de ${\mathcal T}_K$) est r\'eduit \`a
$K_\infty = k_\infty$.
Soient $H_k^{\rm pr}$ et $H_K^{\rm pr}$ 
les pro-$p$-extensions ab\'eliennes $p$-ramifi\'ees (i.e., non ramifi\'ees 
en dehors de $p$) maximales, de $k$ et $K$.
Dans le sch\'ema, $H_k$ est le $p$-corps de classes 
de Hilbert de $k$ et, comme $p$ est non ramifi\'e dans $k$,
$H_k\cap k_\infty =k$ et il existe une 
extension $F$ de $k$, contenant  $H_k$, telle que $H_k^{\rm pr}$ 
soit le compos\'e direct de $F$ et $k_\infty$ sur $k$. On pose 
$\Gamma_\infty = {\rm Gal}(H_k^{\rm pr}/F)$.

\smallskip
Les groupes ${\mathcal A}_k := {\rm Gal}(H_k^{\rm pr}/k)$ 
et ${\mathcal A}_K := {\rm Gal}(H_K^{\rm pr}/K)$ sont
des $\Z_p$-modules de sous-modules de torsion
${\mathcal T}_k$ et ${\mathcal T}_K$. Comme $k$ 
est r\'eel, $\order {\mathcal T}_k$ est donn\'e, 
sous la conjecture de Leopoldt, par la formule du r\'esidu 
de la fonction $\zeta$ $p$-adique de~$k$ (\cite{Co}, \cite {Col}, 
\cite{Se}). Mais on peut affirmer, comme nous l'avons expliqu\'e dans
\cite{Gra7}, que l'analytique $p$-adique classique n'apporte 
actuellement aucune information, aussi nous en resterons aux
caract\'erisations {\it arithm\'etiques} de ${\mathcal T}_k$.

\smallskip
Soit $\overline  E_k$ l'adh\'erence de l'image diagonale $\iota(E_k)$ 
de $E_k$ dans $U_k$~; d'apr\`es le corps de classes, on a 
${\rm Gal}(H_k^{\rm pr}/H_k) \simeq U_k/\overline  E_k$. 
On v\'erifie, puisque ${\rm rg}_{\Z_p}(\overline  E_k)=d-1$,
que ${\rm tor}_{\Z_p}(U_k/\overline  E_k) = U_k^*/\overline  E_k$
o\`u $U_k^*$ est le noyau de la norme absolue.
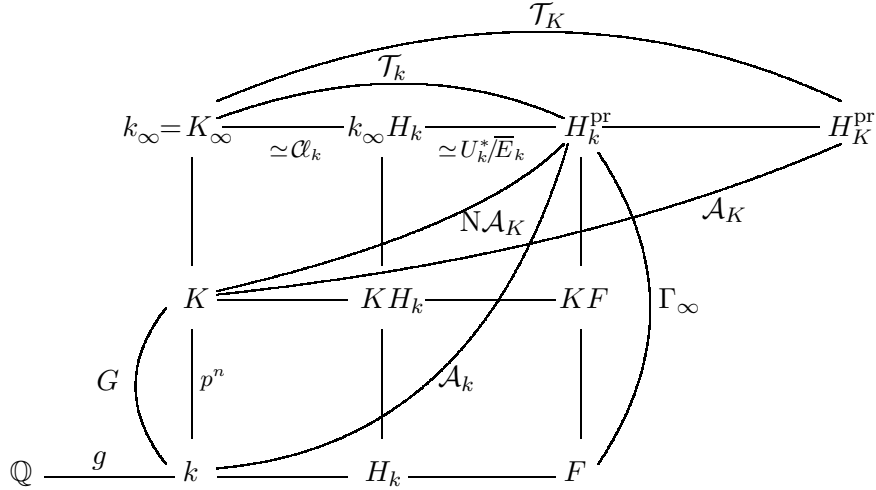
\unitlength=1.15cm 
$$\vbox{\hbox{\hspace{-3.8cm}  \begin{picture}(9.5,5.8)
\put(6.2,4.50){\line(1,0){1.5}}
\put(6.2,2.50){\line(1,0){1.5}}
\put(6.0,0.45){\line(1,0){1.7}}
\put(8.25,4.50){\line(1,0){2.5}}
\put(3.85,4.50){\line(1,0){1.4}}
\put(3.8,2.50){\line(1,0){1.5}}
\put(3.8,0.45){\line(1,0){1.5}}
\bezier{650}(3.8,4.8)(7.6,6.4)(11.0,4.8)
\put(7.4,5.7){${\mathcal T}_{K}$}
\bezier{550}(3.8,4.6)(6.0,5.4)(7.8,4.6)
\put(5.65,5.1){${\mathcal T}_{k}$}
\put(3.50,2.9){\line(0,1){1.25}}
\put(3.50,0.9){\line(0,1){1.25}}
\put(5.7,2.9){\line(0,1){1.25}}
\put(5.7,0.9){\line(0,1){1.25}}
\put(8.0,0.9){\line(0,1){1.25}}
\put(8.0,2.9){\line(0,1){1.25}}
\put(4.4,4.15){\footnotesize$\simeq \! \Cl_k$}
\bezier{700}(3.8,2.56)(8.0,3.0)(11.0,4.3)
\put(9.4,3.45){${\mathcal A}_{K}$}
\bezier{600}(3.8,0.55)(6.8,0.8)(7.85,4.3)
\put(6.35,1.5){${\mathcal A}_k$}
\bezier{500}(3.8,2.6)(6.8,3.3)(7.8,4.3)
\put(6.6,3.3){${\rm N}{\mathcal A}_{K}$}
\put(10.85,4.4){$H_{K}^{\rm pr}$}
\put(5.3,4.4){$k_\infty^{} H_k$}
\put(7.8,4.4){$H_k^{\rm pr}$}
\put(6.35,4.15){\footnotesize$\simeq \! U_k^*\!/\! \overline  E_k$}
\put(2.7,4.4){$k_\infty^{}\! \!= \!K_\infty^{}$}
\put(5.45,2.4){$K H_k$}
\put(7.75,2.4){$K F$}
\put(5.5,0.4){$H_k$}
\put(7.8,0.4){$F$}
\put(3.4,2.4){$K$}
\put(3.4,0.40){$k$}
\put(1.4,0.40){$\Q$}
\put(1.8,0.45){\line(1,0){1.5}}
\put(2.35,0.6){$g$}
\put(3.6,1.5){\footnotesize $p^n$}
\bezier{600}(8.2,0.6)(9.4,2.4)(8.2,4.2)
\put(8.9,2.4){$\Gamma_\infty$}
\put(2.4,1.45){$G$}
\bezier{500}(3.2,2.4)(2.5,1.5)(3.2,0.6)
\end{picture}   }} $$
\unitlength=1.0cm
Rappelons ce qui r\'esulte de la conjecture de Leopoldt et 
qui justifie le sch\'ema~:

\begin{theorem}(\cite[\S\,4]{Gra4})).
Soient $k$ un corps de nombres totalement r\'eel et $p\geq 2$ quelconques.
Sous la conjecture de Leopoldt pour $p$ dans $k$ on a en toute g\'en\'eralit\'e
les suites exactes~:
\begin{equation*}
\begin{aligned}
&1 \too {\rm tor}_{\Z_p}(U_k/\overline  E_k) \tooo  {\mathcal T}_k \tooo 
{\rm Gal}(k_\infty H_k/k_\infty) \too 1, \\
&1 \to {\rm tor}_{\Z_p}(U_k)/ \iota(\mu_k)
\too  {\rm tor}_{\Z_p}(U_k/\overline  E_k) \too  
{\rm tor}_{\Z_p} \! ({\rm log}(U_k)/{\rm log}(\overline  E_k))  \to 1,
\end{aligned}
\end{equation*}
o\`u $\mu_k$ est le groupe des racines de l'unit\'e d'ordre 
puissance de $p$ de $k$, et o\`u l'on a pos\'e
${\rm tor}_{\Z_p} ({\rm log}(U_k)/{\rm log}(\overline  E_k)) =: {\mathcal R}_k$
qui est appel\'e le r\'egulateur $p$-adique normalis\'e de~$k$.
\end{theorem}

\begin{corollary}\label{sefond}
Ces suites exactes se r\'esument, dans le cas $p>2$ totalement d\'ecompos\'e
(o\`u ${\rm tor}_{\Z_p}(U_k)=1$), au moyen de la suite exacte~:
\begin{equation}\label{se1}
1 \too U^*_k/\overline  E_k \simeq {\mathcal R}_k \tooo {\mathcal T}_k 
\tooo \Cl_k \too 1.
\end{equation}
\end{corollary}

Introduisons les symboles d'Artin $\big( \frac{H_k^{\rm pr}/k}{\cdot} \big)$ et 
$\big( \frac{H_{K}^{\rm pr}/K}{\cdot} \big)$, respectivement
sur ${\mathcal J}_k :=  I_k \otimes \Z_p$ et ${\mathcal J}_K :=I_{K} \otimes \Z_p$,
o\`u $I_k$ et $I_{K}$ sont les groupes des id\'eaux \'etrangers \`a $p$ de $k$ et $K$.
Leurs images sont les groupes de Galois ${\mathcal A}_k$ et ${\mathcal A}_{K}$~;
leurs noyaux sont les groupes d'id\'eaux principaux 
infinit\'esimaux ${\mathcal P}_{k, \infty} \subset {\mathcal J}_k$ et 
${\mathcal P}_{K, \infty}  \subset {\mathcal J}_K$, o\`u ${\mathcal P}_{k, \infty}$ est 
l'ensemble des id\'eaux principaux $(x_\infty)$ o\`u $x_\infty \in k^\times  \otimes \Z_p$ 
est \'etranger \`a $p$ et d'image diagonale triviale dans $U_k$, 
et de m\^eme avec $K$ (\cite[Theorem III.2.4, Proposition III.2.4.1]{Gra1} et
\cite[\S\,2]{J1}). 

\smallskip
Le lien entre les normes d'id\'eaux dans $K/k$ et le groupe 
de torsion ${\mathcal T}_k$ est donn\'e par le r\'esultat suivant~:

\begin{theorem} \label{delta}
Soit  ${\mathfrak A} \in I_{K}$ (id\'eal ordinaire vu dans ${\mathcal J}_K$
pour $K=k_n$).
Alors il existe des id\'eaux ${\mathfrak a}, {\mathfrak t} \in {\mathcal J}_k$
et $(x_\infty) \in {\mathcal P}_{k, \infty}$, tels que~:

\smallskip
\centerline{${\rm N}_{K/k}({\mathfrak A}) = {\mathfrak a}^{p^n} \!\cdot\! 
{\mathfrak t} \cdot (x_\infty), \ \ 
\hbox{ avec {\footnotesize $\Big(\displaystyle  \frac{H_k^{\rm pr}/k}
{{\mathfrak a}} \Big)$}} \in \Gamma_\infty,\ 
 \hbox{{\footnotesize $\Big(\displaystyle  \frac{H_k^{\rm pr}/k}
{{\mathfrak t}} \Big)$}} \in {\mathcal T}_{k}$.}
 
\smallskip\noindent
Pour $n \gg 0$, ${\mathfrak a}^{p^n}$ est 
principal de la forme $(\alpha)$, $\alpha \in k^\times \otimes \Z_p$,
o\`u l'image diagonale $\iota(\alpha)$ de $\alpha$ dans $U_k$ v\'erifie
$\iota(\alpha) \equiv 1 \pmod {p^{n'}}$ pour $n' \to \infty$ avec $n$, et 
${\mathfrak t}$ est d'ordre fini modulo ${\mathcal P}_{k, \infty}$. 
\end{theorem}

\begin{proof}
L'application norme arithm\'etique ${\rm N}_{K/k}$, d\'efinie sur 
${\mathcal J}_K$, induit via les symboles d'Artin
la restriction (ou projection) ${\mathcal A}_{K} \to {\mathcal A}_k$ 
(not\'ee encore ${\rm N}_{K/k}$) qui s'exprime 
par la suite exacte suivante (cf. Sch\'ema)~:
\begin{equation}\label{suite}
1 \to {\rm Gal}(H_{K}^{\rm pr} / H_{k}^{\rm pr}) \too {\mathcal A}_{K} 
 \mathop {\tooo}^{ {\rm N}_{K/k}\ }{\rm N}_{K/k}({\mathcal A}_{K}) =  
\Gamma_\infty^{p^n} \oplus {\mathcal T}_{k} \to 1.
\end{equation}

\noindent
Il en r\'esulte que le symbole d'Artin de ${\rm N}_{K/k}({\mathfrak A})$
dans ${\mathcal A}_k$ se d\'ecompose de fa\c con unique sur
$\Gamma_\infty^{p^n} \oplus {\mathcal T}_{k}$ sous la forme
$\Big(\Frac{H_k^{\rm pr}/k}{{\rm N}_{K/k}({\mathfrak A})} \Big)
= \Big(\Frac{H_k^{\rm pr}/k}{{\mathfrak a}} \Big)^{p^n}\!\!\!\cdot
\Big(\Frac{H_k^{\rm pr}/k}{{\mathfrak t}} \Big)$,
et ainsi ${\rm N}_{K/k}({\mathfrak A})$ peut s'\'ecrire dans ${\mathcal J}_k$
comme indiqu\'e dans l'\'enonc\'e. $\square$
\end{proof}

Le fait que les ${\rm N}_{K/k}({\mathfrak A})$
interviennent de fa\c con cruciale pour la conjecture de Greenberg
est justifi\'e, dans la sous-section suivante, au moyen du \go calcul
du $p$-groupe des classes de $K=k_n$\gf par l'algorithme de d\'evissage 
classique qui n'utilise que les propri\'et\'es \'el\'ementaires de ces normes,
la conjecture de Greenberg \'etant \'equivalente au fait que ces algorithmes
sont born\'es ind\'ependamment de $n$ (Th\'eor\`eme \ref{thmp}).
Il suffit alors d'hypoth\`eses naturelles gouvernant ces normes
pour en d\'eduire cette propri\'et\'e.

\smallskip
Or on verra que ${\rm N}_{K/k}({\mathfrak A}) = 
(\alpha \cdot x_\infty) \cdot {\mathfrak t}$ ne
d\'epend que de ${\mathfrak t}$ sur le plan $p$-groupe de classes
de $k$ (${\rm N}_{K/k}({\mathfrak A})$ et ${\mathfrak t}$ d\'efinissent la 
m\^eme classe) et sur le plan $p$-adique 
($\iota(\alpha \cdot x_\infty)=\iota(\alpha)$
est arbitrairement proche de $1$ et ${\mathfrak t}$ d'ordre fini
modulo ${\mathcal P}_{k, \infty}$). Lorsque ${\mathfrak t}$ est 
principal, le lien subtil entre ${\mathfrak t}$ et le r\'egulateur 
${\mathcal R}_{k}$, est pr\'ecis\'e dans la Remarque \ref{norm} 
suivant le Th\'eor\`eme \ref{thmp}. 

\smallskip
Tout ceci est essentiel car ${\rm N}_{K/k}({\mathfrak A})$ ne 
d\'epend alors (via le Corollaire \ref{sefond}) que des invariants $\Cl_k$ et 
${\mathcal R}_{k}$ du groupe {\it fini} ${\mathcal T}_{k}$ qui devient,
quel que soit $K=k_n$, $n \gg 0$, un espace probabilis\'e explicite
(num\'eriquement parlant).

\section{Filtration des $\Cl_{k_n}$ -- Facteur classes 
et facteur normique}\label{sub2}
Soit $K := k_n \subset k_\infty$ de degr\'e $p^n$ sur $k$ et soit 
$G:=G_n:= {\rm Gal}(K/k) =: \langle \sigma \rangle$.
Dans le cadre de l'algorithme g\'en\'eral de calcul du $p$-groupe 
des classes $\Cl_K$ de $K$ par d\'evissage, 
on dispose d'une filtration, au moyen de sous-groupes
$M_i^n := \cl_K({\mathcal I}_i^n)$, $i \geq 0$, 
${\mathcal I}_i^n \subset I_K$, ainsi d\'efinie avec 
$M^n := \Cl_{K}$ et $M^n_0:= 1$ (d'apr\`es \cite[\S\,6.1]{Gra3})~:

\begin{definition} \label{def1} Pour $n \geq 1$ fix\'e, $(M^n_i)_{i \geq 0}$ 
est la $i$-suite de sous-${G}$-modules de $M^n$ d\'efinie par
$M^n_{i+1}/M^n_i := (M^n/M^n_i)^{G}$, \  pour $0\leq i \leq m_n-1$,
o\`u $m_n$ est le plus petit entier $i$ tel que $M^n_i = M^n$ 
(i.e., tel que $M^n_{i+1} = M^n_i$).
\end{definition}

On a alors classiquement~:

\begin{theorem} \label{filtration}
La filtration pr\'ec\'edente a les propri\'et\'es suivantes~:

\smallskip
 (i) Pour $i=0$, $M^n_1 = (M^n)^{G}$ (groupe des classes ambiges
dans $K/k$). 

\smallskip
 (ii) On a 
$M^n_i= \{c \in M^n, \, c^{(1-\sigma)^i} =1 \}$, pour tout $i \geq 0$.

\smallskip
 (iii) Pour $n$ fix\'e, la $i$-suite des $\order  (M^n_{i+1}/M^n_i)$, $0 \leq i \leq m_n$, est 
{\it d\'ecrois\-sante} vers $1$ et major\'ee par $\order  M_1^n$  en raison des injections~:
$$M^n_{i+1}/M^n_i \hookrightarrow M^n_i/M^n_{i-1}
\hookrightarrow \cdots \hookrightarrow M^n_2/M^n_1 \hookrightarrow M_1^n$$
d\'efinies par l'op\'eration de $1-\sigma$.
\end{theorem}

On en d\'eduit que les groupes ${\rm N}_{K/k}({\mathcal I}_i^n)$, 
qui repr\'esentent ${\rm N}_{K/k}(M^n_i) \subseteq \Cl_k$,
sont engendr\'es, modulo des $(\alpha)$ \go quasi-infinit\'esi\-maux\gf 
(i.e., $\iota(\alpha)$ tr\`es proche de~$1$ dans $U_k$), par des 
${\mathfrak t} \in  {\mathcal I}_{k}$ d'ordre fini modulo 
${\mathcal P}_{k, \infty}$ (Th\'eor\`eme \ref{delta}), et que les groupes
de nombres~:
\begin{equation}\label{lambda}
\Lambda_i^n := \{x \in k^\times,\  (x) \in {\rm N}_{K/k}({\mathcal I}_i^n) \},
\end{equation}
qui contiennent $E_k$, sont obtenus via les id\'eaux principaux $(x)$ de la forme~:
$$(x) = (\alpha \cdot x_\infty)\cdot {\mathfrak t}_{(x)}, $$
o\`u ${\mathfrak t}_{(x)} =: (x')$ est un id\'eal principal 
d'ordre fini modulo ${\mathcal P}_{k, \infty}$.

\smallskip
Rappelons que, pour $n \geq 1$ fix\'e, une g\'en\'e\-ralisation de la formule des
classes ambiges de Chevalley, que nous avons red\'emontr\'ee dans \cite{Gra2}, conduit 
\`a la $i$-suite d'entiers d\'efinie, au moyen des groupes
${\rm N}_{K/k}(M^n_i)$ et $\Lambda_i^n$, par~:
\begin{equation}\label{cr}
\begin{aligned}
& \order  \big( M^n_{i+1} / M^n_i \big)=
\frac{\order  \Cl_{k}}{\order {\rm N}_{K/k}( M^n_i)} \cdot  
\frac{p^{n \cdot (d -1)}}{(\Lambda_i^n :  \Lambda_i^n \cap {\rm N}_{K/k}(K^\times))}, \\
  \hbox{o\`u~: \hspace{0.5cm}}  & \\
& \hspace{1.0cm} \displaystyle  \frac{\order  \Cl_{k}}{\order {\rm N}_{K/k}( M^n_i)}
\hspace{0.5cm} \& \hspace{0.5cm}
\displaystyle  \frac{p^{n \cdot (d -1)}}{(\Lambda_i^n :  \Lambda_i^n \cap {\rm N}_{K/k}(K^\times))}
\end{aligned}
\end{equation}
sont appel\'es respectivement  {\it le facteur classes} et {\it le facteur normique}
\`a l'\'etape~$i$ de l'algorithme dans $k_n$. 
La propri\'et\'e essentielle de ces facteurs est que le premier est trivialement un
diviseur de $\order \Cl_{k}$ tandis que le second est (non trivialement)
un diviseur de $\order {\mathcal R}_k$ \cite[Th\'eor\`eme 4.8 (iii)]{Gra3}.
Ils sont ind\'ependants du choix des \'el\'ements des
${\mathcal I}_i^n$ modulo des id\'eaux principaux de $K$.
Comme $\order \Cl_{K} = \prod_{i=0}^{m_n-1} \order (M_{i+1}^n/M_i^n)$, la conjecture 
de Greenberg revient \`a estimer le nombre de pas $m_n$ de l'algorithme. 
Or on a \`a ce sujet le r\'esultat essentiel suivant \cite[Th\'eor\`eme 6.3]{Gra3},
qui montre d\'ej\`a le r\^ole crucial jou\'e par ${\mathcal T}_k$~:

\begin{theorem} \label{thmp}
Pour tout $n \gg 0$, on a
les in\'egalit\'es suivantes (pour ${\mathcal T}_k \ne 1$)%
\,\footnote{Le cas ${\mathcal T}_k=1$ implique
trivialement $\lambda=\mu=\nu=0$~; c'est m\^eme \'equivalent dans le cas 
$p$-d\'ecompos\'e. De fait, d'apr\`es \cite[Th\'eor\`eme 4.7]{Gra3}, on a
$\order \Cl_{K}^{G} = \order {\mathcal T}_{k}$ pour tout $n \gg 0$.}~:
\begin{equation}\label{inegal}
\lambda \cdot n + \mu \cdot p^n +  \nu\, \geq\,
m_n \, \geq \, \frac{1}{v_p(\order {\mathcal T}_k)} \,
\big (\lambda \cdot n + \mu \cdot p^n +  \nu \big),
\end{equation}
o\`u $\lambda, \mu,  \nu$ sont les invariants d'Iwasawa et 
$v_p$ la valuation $p$-adique.

\smallskip\noindent
Par cons\'equent, on a $\lambda=\mu=0$ si et seulement si le nombre de
pas  $m_n$ de l'algorithme dans $k_n$ est born\'e ind\'ependamment de $n$. 
\end{theorem}

Or $m_n$ (pour $n$ fix\'e) d\'epend de la $i$-progression des deux facteurs
de \eqref{cr} et on constate, en pratique, que sous r\'eserve de 
probabilit\'es naturelles de r\'epartition sur les composantes
$\Cl_k$ et ${\mathcal R}_k$ de ${\mathcal T}_k$, 
en relation avec le Th\'eor\`eme \ref{delta}, chacun des deux facteurs est 
rapidement rendu trivial.

\smallskip
On rappelle les r\'esultats suivants \cite[Lemme 7.1]{Gra3}, 
o\`u $i$ est ici fix\'e~:

 (i) La $n$-suite $\displaystyle  \frac{\order \Cl_{k}}{\order {\rm N}_{k_n/k}( M^n_i) } 
=: p^{c_i^n}$ est {\it croissante} stationnaire \`a une valeur maximale
$p^{c_i^{\infty}} \mid \order  \Cl_k$. 
Ce facteur est rapidement trivialis\'e (i.e., $p^{c_i^{\infty}}=1$ pour 
tout $i \geq i_1$ convenable) sous r\'eserve que les classes 
$\cl_k({\rm N}_{k_n/k}({\mathfrak A})) = \cl_k({\mathfrak t})$ 
(o\`u les ${\mathfrak A}$ sont des {\it id\'eaux al\'eatoires de $k_n$} 
associ\'es \`a l'algorithme et
${\mathfrak t}$ la composante donn\'ee par la relation
du Th\'eor\`eme \ref{delta}), se r\'epartissent de fa\c con uniforme 
dans le groupe des classes de $k$ (i.e., le recouvrent selon
les densit\'es naturelles).

 (ii) La $n$-suite $\displaystyle  \frac{p^{n \cdot (d -1)}}
{(\Lambda_i^n :  \Lambda_i^n \cap {\rm N}_{k_n/k}(k_n^\times))} =: p^{\rho_i^n}$
est {\it croissante} stationnaire  \`a une valeur maximale
$p^{\rho_i^{\infty}} \mid \order {\mathcal R}_k$.
On a $(\Lambda_i^n : \Lambda_i^n \cap {\rm N}_{k_n/k}(k_n^\times)) = 
p^{n \cdot (d -1)}$ (i.e., $p^{\rho_i^n}=1$ pour tout $i \geq i_2$ convenable)
d\`es que les symboles normiques de Hasse de {\it suffisamment} de 
$x \in \Lambda_i^n$ associ\'es \`a l'algorithme engendrent le sous-groupe 
$\Omega(k_n/k)$ de 
$\bigoplus_{{\mathfrak p} \in S_k} I_{\mathfrak p} = G_n^d$
(o\`u les $I_{\mathfrak p}$ sont les groupes d'inertie), form\'e des familles 
{\footnotesize 
$\big(\big(\displaystyle  \frac{x, k_n/k}{{\mathfrak p}} \big)\big)_{{\mathfrak p} \in S_k}$} 
v\'erifiant la formule du produit~; or chaque symbole d\'epend essentiellement du
${\mathfrak p}$-quotient de Fermat ${\mathfrak p}^{\delta_{\mathfrak p}(x)}$ 
de $x$, o\`u l'on a pos\'e~:
\begin{equation}\label{qf}
\hbox{\ \ {\footnotesize $\displaystyle  \frac{x^{p-1}-1}{p}$} $=: \prd_{{\mathfrak p} \in S_k}
{\mathfrak p}^{\delta_{\mathfrak p}(x)} \cdot {\mathfrak b}_{(x)}$, 
$\ \ \delta_{\mathfrak p}(x) \geq 0$,  ${\mathfrak b}_{(x)}$ \'etranger \`a $p$}.
\end{equation}
\noindent
Voir  \cite[Th\'eor\`eme 4.4]{Gra3} pour le calcul de ces symboles.
L'ordre de {\footnotesize $\big(\displaystyle \frac{x, k_n/k}{{\mathfrak p}} \big)$} \'etant  
$p^{n - \delta_{\mathfrak p}(x)}$, ce symbole engendre $G_n$ si 
et seulement si $\delta_{\mathfrak p}(x)=0$, 
a priori de probabilit\'e $1-\frac{1}{p}$, $\delta_{\mathfrak p}(x)\geq r$ 
\'etant de probabilit\'e $\frac{1}{p^r}$. D'o\`u aussi une probable
trivialisation rapide du facteur normique au cours de l'algorithme.

\begin{remark}\label{norm}{\rm 
Nous ne revenons pas sur les r\'esultats de 
\cite[Th\'eor\`emes 4.7, 4.8, 4.10]{Gra3}
montrant que les facteurs classes et normiques s'interpr\`etent via
${\mathcal T}_k$, mais pr\'ecisons l'aspect pratique
et probabiliste pour les futures \'etudes num\'eriques.
Soit $x \in \Lambda_i^n$.
On a $(x) = {\rm N}_{k_n/k}({\mathfrak A})$, 
${\mathfrak A} \in {\mathcal I}_i^n$
(cf. Th\'eor\`eme \ref{filtration} \& \eqref{lambda}),
et d'apr\`es le Th\'eor\`eme \ref{delta}, 
$(x) = (\alpha \cdot x_\infty) \cdot {\mathfrak t}_{(x)}$, 
$\big(\frac{H_k^{\rm pr}/k}{{\mathfrak t}_{(x)}} \big) \in 
{\rm Gal} (H_k^{\rm pr} / k_\infty H_k)$, ce qui conduit \`a 
$x = \alpha \cdot x_\infty \cdot x'$, o\`u $(x')={\mathfrak t}_{(x)}$
(pour tout $n$ assez grand).

\smallskip
Ensuite, puisque ${\rm Gal}(H_k^{\rm pr}/k_\infty \,H_k)
\simeq U_k^*/\overline  E_k$ 
est d'exposant fini $p^e$,  on a $x'{}^{p^e} = x'_\infty\! \cdot \varepsilon'$, 
$\varepsilon' \in E_k \otimes \Z_p$, $x'_\infty$ infinit\'esimal, d'o\`u
${\rm N}_{k/\Q}(\iota(x'))=1$ dans $U_\Q$ et l'image de $x'$ 
est d\'efinie dans $U_k^*/\overline  E_k = {\mathcal R}_{k}$, donc ne prend
qu'un nombre fini de valeurs.

\smallskip
Puisque dans $U_k$ (pour $n$ assez grand) 
$\delta_{\mathfrak p}(x) =
\delta_{\mathfrak p}(\iota(x'))$, nous dirons, par abus, que la famille 
$(\delta_{\mathfrak p}(x))_{{\mathfrak p} \mid p}$, $x \in \Lambda_i^n/E_k$, 
varie dans le domaine (fini) d\'efini par l'ensemble 
des familles $(\delta_{\mathfrak p}(\varepsilon))_{{\mathfrak p} \mid p}$, 
$\varepsilon  \in E_k$~; c'est un invariant canonique du corps $k$. 
Par exemple, dans le cas quadratique, d'unit\'e fondamentale $\varepsilon$, 
pour $\delta_{\mathfrak p}(\varepsilon) = r >0$ donn\'e on \'ecrira que
$\delta_{\mathfrak p}(x E_k) \in [0, r]$ selon des probabilit\'es naturelles.}
\end{remark}

Nous avons formul\'e dans \cite[Hypoth\`ese 7.9]{Gra3} les hypoth\`eses
suivantes, au sujet de la composante ${\mathfrak t}$ 
de ${\rm N}_{K/k}({\mathfrak A})$ d\'efinie au Th\'eor\`eme \ref{delta},
que nous nous proposons de tester num\'eriquement~:

\medskip\noindent
{\bf Hypoth\`eses (H)}
{\it On suppose que les id\'eaux ${\mathfrak A}$ (\'etrangers \`a $p$) 
de $K=k_n$, obtenus au cours de l'algorithme de calcul 
par d\'evissage de $\order \Cl_K$, 
d\'efinissent une variable al\'eatoire, et qu'il en est de m\^eme 
pour la composante ${\mathfrak t}$ des ${\rm N}_{K/k}({\mathfrak A})$. 
Enfin on suppose que les symboles d'Artin $\big(\frac{F/k}{{\mathfrak t}} \big)$ 
se r\'epartissent uniform\'ement dans ${\rm Gal}(F/k) \simeq {\mathcal T}_{k}$.}

\medskip
Le terme \go uniform\'ement\gf doit \^etre compris comme une 
propri\'et\'e de recouvrement selon les densit\'es (ou probabilit\'es) 
naturelles en $\frac{1}{p^r}$.
Si ces hypoth\`eses sont v\'erifi\'ees, ceci a les cons\'equences suivantes~:

\medskip
 (i) Les classes des id\'eaux ${\mathfrak t}$ sont r\'eparties uniform\'ement 
dans $\Cl_k$.

\smallskip
 (ii) Dans le cas principal ${\mathfrak t}=(x')$, les images des $x'$ dans 
$U_k^*/\overline E_k = {\mathcal R}_{k}$ parcourent uniform\'ement 
cet ensemble fini.

\smallskip
 (iii) L'heuristique principale de \cite{Gra3}
est que les nombres de pas $m_n$ des algorithmes d\'ependent grosso 
modo de lois binomiales sur l'espace de probabilit\'e fini ${\mathcal T}_{k}$
et sont {\it presque s\^urement uniform\'ement born\'es} lorsque $n \to \infty$. 

\medskip
Les limites $p^{c_i^\infty + \rho_i^\infty}$ des 
$n$-suites $\order \big( M^n_{i+1} / M^n_i \big)$, $n\to\infty$,
constituent une $i$-suite {\it d\'ecroissante stationnaire} 
de diviseurs de $\order {\mathcal T}_{k}$ \cite[Lemme 7.2]{Gra3}.
Or si le diviseur limite est diff\'erent de $1$, 
c'est que l'hypoth\`ese de r\'epartition pr\'ec\'edente n'est pas v\'erifi\'ee, 
ce qui, dans un ensemble fini, suppose l'exis\-tence d'une \go condition \'etrange\gf 
au niveau des composantes ${\mathfrak t}$ successivement obtenues, ce 
que la pratique num\'erique que nous allons mettre en \oe uvre devrait 
rendre absurde~; en effet, ceci voudrait dire pratiquement
que si $\lambda \geq 1$ ou $\mu \geq 1$ alors, pour $n$ fix\'e
{\it arbitrairement grand}, l'algorithme de d\'evissage \go bouclerait\gf 
$O(\lambda \cdot n + \mu \cdot p^n)$ fois de suite (in\'egalit\'es \eqref{inegal}), 
ce qui constitue un non-sens num\'erique mais sugg\`ere l'immense difficult\'e
pour \'etablir une contradiction effective conduisant \`a une preuve 
(certainement plus analytique qu'alg\'ebrique) de la conjecture de Greenberg.

\section{Calcul des principaux invariants -- Classes logarithmiques}\label{sec2}

On se place dans le cas quadratique r\'eel $k=\Q(\sqrt D)$, 
d'unit\'e fondamentale $\varepsilon$, avec $p>2$ 
d\'ecompos\'e, et on utilise des programmes PARI \cite{P}
pour les calculs. Ici $\Omega(k_n/k) \simeq \Z/p^n \Z$ et les questions
normiques ne d\'ependent que des $\delta_{\mathfrak p}(x)$. Mais
$\delta_{\mathfrak p}(x)$,
not\'e $\delta_p(x)$, ne d\'epend pas de ${\mathfrak p} \in S_k$ 
lorsque {\it $x$ est \'etranger \`a~$p$} \cite[D\'efinition 4.1 \& \S\,5.1]{Gra3}.
Le contexte non trivial est $\delta_p(\varepsilon)>0$ (sinon 
${\mathcal R}_k=1$ et seul le facteur classes est concern\'e).
Ensuite, pour les $x$, de l'algorithme,
$\delta_p(x)$ ne d\'epend que de l'id\'eal $(x)$ si 
$\delta_p(x) < \delta_p(\varepsilon)$~; en effet,
 si $\delta_p(\varepsilon)=r$ et $\delta_p(x)<r$, alors 
$\delta_p(x \cdot \varepsilon') = \delta_p(x)$ quelle que soit 
$\varepsilon' \in E_k$. Enfin le facteur normique se trivialise
d\`es qu'on obtient $x \in \Lambda_i^n$ avec $\delta_p(x)=0$.

\smallskip
Rappelons que les $x \in \Lambda_i^n$, par nature normes 
d'id\'eaux \'etrangers \`a $p$ dans $k_n/k$, sont partout normes locales 
en dehors de $p$, auquel cas, ${\rm N}_{k/\Q}(x) \equiv 1 \pmod {p^{n+1}}$
(en effet, ${\rm N}_{k/\Q}\big ((x) \big)$ peut s'\'ecrire 
${\rm N}_{\Q_n/\Q}({\mathfrak B})$ pour un id\'eal ${\mathfrak B}$ 
de $\Q_n$, or le groupe de normes associ\'e \`a $\Q_n/\Q$ par le 
corps de classes est l'ensemble des $(1+  a \cdot  p^{n+1})$, $a \in \Z$). 
En pratique il n'est pas n\'ecessaire de prendre $n$ tr\`es grand car les 
$\delta_p(\varepsilon)$ sont limit\'es, et dans le cas quadratique,
le r\'egulateur $p$-adique normalis\'e est ${\mathcal R}_k \sim 
\frac{1}{p} \,{\rm log}(\varepsilon) \sim p_{}^{\delta_p(\varepsilon)}$
(\'egalit\'es \`a un facteur unit\'e $p$-adique pr\`es)~; 
on fixe $n_0:=n+1$ (on a alors $K=k_{n}$ et des congruences modulo 
$p^{n_0}$ pour les calculs des symboles normiques). 
Mais on fera en sorte que $n > \delta_p(\varepsilon)$.

\subsection{\bf Programme de calcul de $h$, $\delta_p(\varepsilon)$, 
$\delta_{\mathfrak p}(\eta_p)$ -- Condition suffisante}\label{2sub1}

Il donne la liste des discriminants $D$ pour lesquels $k=\Q(\sqrt D)$ r\'epond \`a 
certaines sp\'ecifications destin\'ees \`a tester l'influence des param\`etres
suivants~: nombre de classes $h$, valeurs de $\delta_p(\varepsilon)$ et 
$\delta_{\mathfrak p}(\eta_p)$, o\`u $\eta_p$ est une 
$S_k$-unit\'e fondamentale, donn\'ee par ${\mathfrak p}^{h_0}=
(\eta_p)$ o\`u $h_0$ est l'ordre de la classe de~${\mathfrak p}$, et qui 
n'est utilis\'ee que pour tester la condition suffisante \cite[Th\'eor\`eme 3.4]{Gra3} 
conduisant \`a $\lambda = \mu =0$ (condition qui \'equivaut \`a 
la trivialit\'e du groupe des classes logarithmiques $\widetilde {\Cl_k}$ 
\cite[Th\'eor\`eme 17]{J2}, invariant sur lequel nous reviendrons 
au \S\,\ref{log})~; elle suppose la r\'ealisation des deux conditions suivantes~:

\smallskip
 (i) $\Cl_k^{S_k}=1$ (i.e., ${\mathfrak p}$ engendre le $p$-groupe 
des classes de $k$), ce qui \'equivaut \`a $v_p(h_0) = v_p(h)$,

\smallskip
 (ii)  $\delta_p(\varepsilon) = 0 \ \ \hbox{ou}  \ \ \delta_{\mathfrak p}(\eta_p) = 0$.

\medskip
Si (i) (resp. (ii)) n'a pas lieu, le programme indique la nature du probl\`eme par 
PB-CLASSES (resp. PB-NORMIQUE). En l'absence des deux alertes, la 
condition suffisante a lieu et $k$ v\'erifie la conjecture de Greenberg,
ce qui fournit beaucoup d'exemples avec des invariants $\Cl_k$ et 
${\mathcal R}_k$ non triviaux.

\smallskip
On peut fixer un ordre de grandeur \`a $\delta_p(\varepsilon)$ 
pour les discriminants retenus
en \'ecrivant dans le programme ${\sf zmax}=1/p^g$
pour $\delta_p(\varepsilon) \geq g \geq 0$, et de m\^eme on peut imposer 
des conditions sur $h$ comme par exemple $h\equiv 0 \pmod p$ qui se
traduit par ${\sf vh}>0$. L'unit\'e $\varepsilon$ est donn\'ee 
dans {\sf E} et $\eta_p$ dans {\sf Eta}.
On doit enfin choisir le nombre premier ${\sf p}=p$ 
et les bornes ${\sf bD}$, ${\sf BD}$ du discriminant~:

\footnotesize
\begin{verbatim}
=============================================================================
{p=3;bD=2;BD=5*10^5;n=8;n0=n+1;zmax=1/p^2;y=x;
for(D=bD,BD,e=valuation(D,2);M=D/2^e;if(core(M)!=M,next);
if((e==1||e>3)||(e==0 & Mod(M,4)!=1)||(e==2 & Mod(M,4)==1)
|| kronecker(D,p)!=1,next);Q=x^2-D;K=bnfinit(Q,1);
h=component(component(bnrinit(K,1),5),1);vh=valuation(h,p);if(vh>=1,
E=component(component(component(K,8),5),1);Su=bnfsunit(K,idealprimedec(K,p));
pi1=component(component(Su,1),1);pi2=component(pi1,2)*x-component(pi1,1);
Pi1=pi1^n0;Pi2=pi2^n0;Z=bezout(Pi1,Pi2);U1=component(Z,1);U2=component(Z,2);
P=y^2-Mod(D,p^n0);Y=Mod(y,P);x=Y;A1=eval(U1);A2=eval(U2);
B1=eval(Pi1);B2=eval(Pi2);b1=eval(pi1);b2=eval(pi2);e=eval(E);
XPpi=Mod(A1*B1+A2*B2*b2,P);XPe=Mod(A1*B1+A2*B2*e,P);x=y;
hs=norm(Mod(pi1,Q));h0=valuation(hs,p);vh0=valuation(h0,p);delta=vh-vh0;
npi=norm(XPpi)^(p-1);ne=norm(XPe)^(p-1);zpi=znorder(npi)/p^n;
ze=znorder(ne)/p^n;if(ze<=zmax,if(delta!=0,print("PB-CLASSES"));
if(zpi+ze<1,print("PB-NORMIQUE"));print("D=",D," h=",h," E=",E);
print("p=",p," Eta=",pi1);print(zpi," ",ze);print(" "))))}
=============================================================================
\end{verbatim}
\normalsize

Les valeurs $\Frac{1}{p^{\delta_{\mathfrak p}(\eta_p)}}$ et 
$\Frac{1}{p^{\delta_p(\varepsilon)}}$ sont donn\'ees dans ${\sf zpi}$
et ${\sf ze}$.

\smallskip
D\`es que $p$ cro\^it, il y a rar\'efaction des cas exceptionnels~; 
par exemple pour $p=11$, $D \leq 3 \cdot 10^5$, $h\equiv 0 \pmod {11}$,
${\sf zmax}=\frac{1}{11^2}$, on obtient au total 4 cas~:

\footnotesize
\begin{verbatim}
PB-NORMIQUE
D=73217 h=11 E=4007500*x - 1084374999
p=11 Eta=441257*x - 119399338
1/11  1/121

D=83689 h=11 E=8962870747239371437765*x - 2592873462296714584831032
p=11 Eta=-5270913810*x - 1524825351767
1     1/1331

D=201997 h=11 E=1781*x + 800454
p=11 Eta=-64391/2*x - 28959651/2
1     1/14641

D=265681 h=44 E=2400852*x - 1237501225
p=11 Eta=2244943875203844650892*x - 1159999283951803336111535
1     1/121
\end{verbatim}

\normalsize
\subsection{\bf Comparaison avec le groupe des classes logarithmiques $\widetilde {\Cl_k}$}
\label{log}
Le $p$-groupe des classes logarithmiques a \'et\'e introduit 
par Jaulent \cite{J1} et utilis\'e pour la conjecture de Greenberg, 
en ceci qu'il donne la condition n\'ecessaire 
et suffisante suivante, sous la seule conjecture de Leopoldt~:

\begin{theorem}  (\cite[Th\'eor\`eme 7, \S\,1.4]{J2}).
Le corps totalement r\'eel $k$ v\'erifie la conjecture de Greenberg 
si et seulement si son groupe des classes logarithmiques 
$\widetilde {\Cl_k}$ capitule dans $k_\infty$.
\end{theorem}

Ceci est la g\'en\'eralisation de la condition suffisante du \S\,\ref{2sub1}
disant que $\widetilde {\Cl_k} = 1$ entra\^ine la conjecture de Greenberg.
Bien que non asymptotique, ce crit\`ere est non effectif quant au 
$n_0$ \`a partir duquel l'application d'extension des classes
$\widetilde  j_{k_{n_0}/k} : \widetilde {\Cl_k} \too \widetilde {{\Cl}_{k}}{}_{n_0}^{}$
est d'image nulle et il serait int\'eressant de faire un rapprochement 
avec l'algorithme de d\'evissage des $\Cl_{k_n}$.

On d\'efinit les $p$-groupes $\widetilde {\Cl_k}$, $\widetilde {\Cl_k}^{\![p]}$, et 
$\Cl'_k := \Cl^{S_k}_k  := 
\Cl_k / \langle \cl_k({\mathfrak p}),\  {\mathfrak p} \in S_k \rangle$, 
par la suite exacte~:
$$1\too \widetilde {\Cl_k}^{\![p]} \tooo \widetilde {\Cl_k} \tooo \Cl'_k  \too 1, $$
qui permet la comparaison avec $\Cl'_k$, le groupe des $S_k$-classes
de $k$ associ\'e au corps de Hilbert $p$-d\'ecompos\'e, dont la nullit\'e 
est la premi\`ere partie de la condition suffisante du \S\,\ref{2sub1} 
pour la conjecture de Greenberg. 
La seconde ($\delta_p(\varepsilon)=0 \  \hbox{ou} \ \delta_{\mathfrak p}(\eta_p)=0$),
sous la premi\`ere, \'equivaut donc \`a la nullit\'e de 
$\widetilde {\Cl_k} = \widetilde {\Cl_k}^{\![p]}$.
Pour $k$ fix\'e et $p\gg0$, il est clair que $\Cl'_k=1$ et que 
la nullit\'e de $\widetilde {\Cl_k}$ \'equivaut \`a
$\delta_p(\varepsilon) = 0 \  \hbox{ou} \ \delta_{\mathfrak p}(\eta_p) = 0$.

\smallskip
En utilisant la fonction ${\sf bnflog(K,p)}$ de PARI, d\'ecrite dans \cite{BJ}, on peut 
calculer la structure du $p$-groupe des classes logarithmiques. Pour le cas des corps 
quadratiques r\'eels (mais sans l'hypoth\`ese de $p$-d\'ecomposition que 
l'on peut rajouter via la condition $(\frac{D}{p}) = 1$ sur $D$) 
on obtient le programme suivant (ne retenant que les cas o\`u $\widetilde {\Cl_k} \ne 1$)~:

\footnotesize
\begin{verbatim}
============================================================================
{p=3;for(D=10^4,10^4+10^3,e=valuation(D,2);M=D/2^e;if(core(M)!=M,next);
if((e==1||e>3)||(e==0 & Mod(M,4)!=1)||(e==2 & Mod(M,4)==1),next);
P=x^2-D;K=bnfinit(P,1);H=bnflog(K,p);if(component(H,1)!=[],print(D," ",H)))}
============================================================================
\end{verbatim}
\normalsize

Donnons les courts extraits suivants de cas non triviaux pour $p=3$ et $5$
(selon les notations $[[\widetilde {\Cl_k}], \ [\widetilde {\Cl_k}^{\![p]}], \  [\Cl'_k]]$ de \cite[\S\,4]{BJ})~:

\footnotesize
\begin{verbatim}
p=3
 D     structures         D     structures         D     structures
10040 [[3], [],  [3]]    10585 [[3], [3], []]     10849 [[27], [27], []]
10060 [[3], [3], []]     10636 [[3], [3], []]     10865 [[3],  [],   [3]]
10077 [[3], [],  [3]]    10641 [[3], [],  [3]]    10889 [[3],  [],   [3]]
10153 [[3], [3], []]     10661 [[3], [],  [3]]    10904 [[3],  [],   [3]]
10172 [[3], [],  [3]]    10664 [[3], [],  [3]]    10929 [[3],  [],   [3]]
10213 [[3], [3], []]     10712 [[3], [],  [3]]    10941 [[3],  [],   [3]]
10301 [[3], [],  [3]]    10721 [[9], [],  [9]]    10949 [[3],  [],   [3]]
10353 [[3], [],  [3]]    10733 [[3], [],  [3]]    10972 [[3],  [3],  []]
10357 [[3], [3], []]     10812 [[3], [],  [3]]    10997 [[3],  [],   [3]]
10457 [[3], [],  [3]]    10844 [[3], [],  [3]]    

p=5
 D     structures           D     structures          D      structures
10284 [[25], [25], []]     10408 [[5], [],  [5]]     10649 [[5], [5], []]
10301 [[5],  [5],  []]     10561 [[5], [5], []]      10821 [[5], [5], []]
10396 [[5],  [5],  []]     10613 [[5], [],  [5]]     10885 [[5], [],  [5]]
\end{verbatim}
\normalsize

Pour $p=5$, il y a rapidement rar\'efaction des cas non triviaux, le cas de
$D= 10284$ \'etant assez exceptionnel.

\smallskip
Quant \`a $p=29$, on trouve par exemple
$D = 4 \cdot 683, 4 \cdot 890, 4 \cdot 1271, 4349$, 
$4 \cdot 7858$, pour lesquels $\widetilde {\Cl_k} = \widetilde {\Cl_k}^{\![p]} \simeq \Z/ 29 \Z$.

\smallskip
On constate plus g\'en\'eralement l'identit\'e des r\'esultats num\'eriques
donn\'es par le programme pr\'ec\'edent avec ceux donn\'es dans 
l'importante table (issue de \cite[\S\,5.2]{Gra3}, cas $p$-d\'ecompos\'e)~: 
\url{https://www.dropbox.com/s/tcqfp41plzl3u60/R}

\smallskip\noindent
o\`u l'indication de au moins l'une des mentions 
\go PB-CLASSES\gf (i.e., $\Cl^{S_k}_k \ne 1$) ou \go PB-NORMIQUE\gf
(i.e., $\delta_p(\varepsilon)\geq 1 \  \& \ \delta_{\mathfrak p}(\eta_p)\geq 1$),
caract\'erise un groupe de classes logarithmiques non trivial, auquel cas
on ne sait pas conclure.

\subsubsection{Remarques sur le groupe des classes logarithmiques}
(i) On notera que pour un corps $k$ totalement r\'eel, d\`es que le corps 
de classes de Hilbert $H_k$ est lin\'eairement disjoint de $K_\infty$,
la {\it $p$-rationalit\'e} de $k$ (i.e., ${\mathcal T}_k=1$) implique la trivialit\'e 
du groupe des classes logarithmiques (i.e., $\widetilde {\Cl_k}=1$), mais que 
la r\'eciproque est largement fausse.

\smallskip
De fa\c con pr\'ecise, sous nos hypoth\`eses ($k$ totalement r\'eel, 
$p$-d\'ecompos\'e, $p>2$), on a (cf. suite exacte \eqref{se1})~:
$$\order {\mathcal T}_k = 
\order \Cl_k \cdot \order {\mathcal R}_k,$$ 
\noindent
tandis que~:
$$\order \widetilde {\Cl_k} =\order \Cl'_k \cdot \order \widetilde {\Cl_k}^{\![p]} ;$$
or $\widetilde {\Cl_k}^{\![p]}$ est isomorphe \`a un quotient de ${\mathcal R}_k$
induit par les $S_k$-unit\'es \cite[Sch\'ema \S\,2.3]{J2}, 
ce qui montre que ${\mathcal T}_k=1$ \'equivaut ici \`a  $\Cl_k=1$ et 
$\delta_p(\varepsilon)=0$ (i.e., ${\mathcal R}_k=1$),
tandis que $\widetilde {\Cl_k}=1$ \'equivaut \`a $\Cl'_k=1 \ \& \ 
(\delta_p(\varepsilon)=0 \  \hbox{ou} \ \delta_{\mathfrak p}(\eta_p)=0)$.

En d\'epit de la notation, $\widetilde {\Cl_k}^{\![p]}$ doit \^etre consid\'er\'e 
comme un \go r\'egulateur logarithmique\gf\!\!, donc
l'invariant essentiel puisque $\Cl'_k=1$ pour $p \gg 0$.

\smallskip
(ii) Donnons des exemples de corps $k=\Q(\sqrt m)$ ($m$ sans facteur carr\'e) tels que
${\mathcal T}_k \ne 1\ \& \ \widetilde {\Cl_k}=1$~:

\smallskip
Corps $k$ tels que ${\mathcal T}_k \ne 1\ \& \ \widetilde {\Cl_k}=1$
avec $\delta_p(\varepsilon) \geq 1 \  \hbox{et} \  \delta_{\mathfrak p}(\eta_p)=0$~:

\smallskip\noindent
\footnotesize
$m=43, 58, 79, 82, 85, 109, 151, 181, 199, 202, 
247, 271, 310, 322, 331, 337, 391$, $406, 457, \ldots$

\smallskip\normalsize
Corps $k$ pour lesquels ${\mathcal T}_k \ne 1\ \& \ \widetilde {\Cl_k}=1$
avec $\delta_p(\varepsilon)=0 \  \hbox{et} \  \delta_{\mathfrak p}(\eta_p) \geq 1$~:

\smallskip \noindent 
\footnotesize
$m=142, 223, 235, 469, \ldots$

\smallskip \normalsize
Corps $k$ pour lesquels  ${\mathcal T}_k \ne 1\ \& \ \widetilde {\Cl_k}=1$
avec $\delta_p(\varepsilon)= \delta_{\mathfrak p}(\eta_p)=0$~:

\smallskip \noindent 
\footnotesize
$m=229, 346, 427, \ldots$

\subsection{\bf Crit\`ere de $p$-rationalit\'e  (i.e., ${\mathcal T}_k=1$)}\label{prat}

\normalsize
Redonnons le programme \cite[Programme I]{Gra5}, qui d\'etermine
la $p$-rationalit\'e d'un corps de nombres arbitraire d\'efini par un
polyn\^ome unitaire irr\'eductible $P \in \Z[x]$, et traitons le cas des corps
quadratiques r\'eels~; pour chaque 
discriminant $D$ on donne le $3$-rang de ${\mathcal T}_k$, 
le test de $3$-rationalit\'e et la structure du $3$-groupe de classes 
de rayon ${\sf 3^{nt}}$, ${\sf nt} \geq 2$ (${\sf nt} =2$ est suffisant pour le test de 
$p$-rationalit\'e et ${\sf nt}$ assez grand donne la structure de ${\mathcal T}_k$)~:

\footnotesize
\begin{verbatim}
==============================================================================
{p=3;bD=2;BD=10^5;nt=9;for(D=bD,BD,e=valuation(D,2);M=D/2^e;
if(core(M)!=M,next);if((e==1||e>3)||(e==0 & Mod(M,4)!=1)||(e==2 & Mod(M,4)==1),
next);P=x^2-D;K=bnfinit(P,1);Kpn=bnrinit(K,p^nt);
Hpn=component(component(Kpn,5),2);L=List;v=component(matsize(Hpn),2);
R=0;for(k=1,v-1,c=component(Hpn,v-k+1);if(Mod(c,p)==0,R=R+1;
listinsert(L,p^valuation(c,p),1)));
if(R>0,print("D=",D," rk(T)=",R," K is not ",p,"-rational ",L));
if(R==0,print("D=",D," rk(T)=",R," K is ",p,"-rational ",L)))}
==============================================================================
\end{verbatim}
\normalsize

On obtient, en se limitant aux corps non $3$-rationnels $3$-d\'ecompos\'es
(rajouter au programme la condition {\sf kronecker(D,p)!=1}) avec
${\rm rg}_3({\mathcal T}_k)>1$, les quelques exemples suivants (o\`u le rang est $2$)~:
\footnotesize
\begin{verbatim}
D=2917  List([9, 3])     D=10636  List([9, 3])      D=14668  List([3, 3])
D=6856  List([3, 3])     D=11293  List([9, 3])      D=15517  List([3, 3])
D=7465  List([9, 9])     D=13273  List([3, 3])      D=15529  List([27, 3])
D=8713  List([9, 3])     D=13564  List([27, 3])     D=15733  List([3, 3])
D=8920  List([3, 3])     D=13861  List([2187, 3])   D=17116  List([9, 3])
D=9052  List([3, 3])     D=14197  List([9, 3])      D=18541  List([3, 3])   
\end{verbatim}
\normalsize
                                        
Par rapport aux corps du point (ii) ci-dessus, on obtient les cas
o\`u l'on a simultan\'ement ${\mathcal T}_k \ne 1\ \& \  \widetilde {\Cl_k}\ne 1$ 
($D=4 \!\cdot\!  67, \ 4 \!\cdot\! 103, \  4 \!\cdot\! 106,\ldots$).

\section{Statistiques sur les symboles d'Artin des
${\rm N}_{K/k}({\mathfrak A})$ -- Exemples}

Nous revenons au principe d'analyse, d\'ecrit Section \ref{sub2}, pour tester les
Hypoth\`eses (H), fin du \S\,\ref{sub2}, sur les normes absolues d'id\'eaux 
{\it \'etrangers \`a} $p$ dans la tour cyclotomique d'un corps quadratique r\'eel 
$p$-d\'ecompos\'e.
On va montrer que l'on peut toujours supposer les id\'eaux ${\mathfrak A}$
premiers pour effectuer les statistiques.

\subsection{\bf Repr\'esentation des classes par des id\'eaux premiers}\label{2sub2}
On suppose d\'esormais (pour les calculs) que $p=3$.
On d\'esire \'etablir des statistiques sur l'influence num\'erique des 
${\rm N}_{K/k}({\mathfrak A})$, ${\mathfrak A} \in{\mathcal I}_i^n$,
de l'algorithme sur le facteur classes et sur le facteur normique de \eqref{cr},
pour $K=k_n$, sachant que l'algorithme d\'etermine des ${\mathfrak A}$
successifs par le biais de \go l'\'equation d'\'evolution 
$(y) = {\mathfrak B}\!\cdot \! {\mathfrak A}^{1-\sigma}$\gf\!\!,
provenant de \go $x={\rm N}_{K/k}(y) \ \&\ (x)={\rm N}_{K/k}({\mathfrak B})$\gf\!\!,
qui sera \'etudi\'ee Section \ref{sec3}.

\smallskip
Toute classe d'id\'eaux de $K$ peut se repr\'e\-senter par un id\'eal premier
${\mathfrak L}$ de $K$, totalement d\'ecompos\'e dans $K/\Q$ 
(th\'eor\`eme de Chebotarev dans $H_K/\Q$).
Comme une relation de la forme ${\mathfrak A} = {\mathfrak L} \cdot (\alpha)$
dans $K$ implique, dans $k$~:
$${\rm N}_{K/k}({\mathfrak A}) = {\rm N}_{K/k}({\mathfrak L}) \cdot( {\rm N}_{K/k}(\alpha)) ,$$
le fait de supposer les id\'eaux ${\mathfrak A}$ premiers est sans cons\'equences
sur l'\'etude statistique des deux facteurs \eqref{cr}~; en effet,
le premier facteur est relatif aux {\it classes des ${\rm N}_{K/k}({\mathfrak A})$} 
aussi repr\'esent\'ees par des ${\rm N}_{K/k}({\mathfrak L})$,
et le second est relatif aux {\it indices normiques} 
$(\Lambda_i^n : \Lambda_i^n \cap {\rm N}_{K/k}(K^\times))$,
o\`u $\Lambda_i^n := \{x \in k^\times,\  (x) \in {\rm N}_{K/k}({\mathcal I}_i^n) \}$
(cf. \eqref{lambda}), qui ne d\'ependent pas du choix des $x \in \Lambda_i^n$, 
modulo $E_k \cdot {\rm N}_{K/k}(K^\times)$.

\smallskip
On consid\`ere donc un grand nombre de premiers $\ell$ totalement d\'ecom\-pos\'es
dans $K/\Q$ (i.e., $\ell \equiv \pm 1 \pmod{3^{n+1}}$ et 
$(\ell) =: {\mathfrak l}\cdot {\mathfrak l}'$ dans $k$). On suppose implicitement 
que ${\mathfrak L}$ et ${\mathfrak l} := {\rm N}_{K/k}({\mathfrak L})$ sont les 
variables al\'eatoires qui \go conduisent l'algorithme\gf \`a chaque \'etape. 

\smallskip
On peut donc supposer ${\rm N}_{K/k}({\mathcal I}_i^n)$ engendr\'e
par des ${\mathfrak l}_j = {\rm N}_{K/k}({\mathfrak L}_j)$ du type pr\'ec\'edent et 
$\Lambda_i^n = \{x \in k^\times, \, (x) \in \langle {\mathfrak l}_j  \rangle_j\}$~; 
le facteur classes d\'epend du sous-groupe 
$\langle \cl_k({\mathfrak l}_j) \rangle_j$ de $\Cl_k$ 
dont la croissance algorithmique, en fonction du nombre de pas, 
est cens\'ee atteindre $\Cl_k$ puisque ${\mathcal I}_{i+1}^n \supset {\mathcal I}_i^n$.

\smallskip
Si $x \in \Lambda_i^n$, on a donc $(x) = \prod_j {\mathfrak l}_j^{e_j}$, $e_j \in \Z$ 
(ce qui repr\'esente une relation entre les classes des ${\mathfrak l}_j$), et le facteur 
normique d\'epend alors des $\delta_3(x)$ et d\'ecro\^it dans la progression de 
l'algorithme puisque $\Lambda_{i+1}^n \supset \Lambda_i^n$. 

\smallskip
Les relations \'etant nombreuses et inconnues pour faire des statistiques, 
on s'int\'eressera dans les sections suivantes
au cas particu\-lier des $x$ qui sont des ${\mathfrak l}$-unit\'es~;
si $r_\ell$ est l'ordre de la classe de ${\mathfrak l}$, on posera
${\mathfrak l}^{r_\ell} = (\eta_\ell)$, o\`u $\eta_\ell$ est une $\ell$-unit\'e 
(d\'efinie modulo $E_k$), puis on calculera $\delta_3(\eta_\ell)$.
Comme ${\rm N}_{k/\Q}(\eta_\ell) = \pm \ell^{r_\ell} \equiv \pm 1 \pmod{3^{n+1+r_\ell}}$, 
on pourra n\'egliger la conjugaison dans $k/\Q$ et travailler avec un unique 
id\'eal premier ${\mathfrak l} \mid \ell$ et une unique $\eta_\ell$.

\smallskip
Ensuite, si l'on constate que les entiers $\delta_3(\eta_\ell)$ 
ont (par rapport \`a $\delta_3(E_k)$) les r\'epartitions attendues,
alors une propri\'et\'e analogue en r\'esultera, a fortiori, pour les $\delta_3(x)$ 
(cens\'es valoir $0$ avec une probabilit\'e non nulle) car
les \go relations\gf particuli\`eres ${\mathfrak l}^{r_\ell} = (\eta_\ell)$ 
ne donnent qu'une partie des $x \in \Lambda_i^n$ susceptibles d'\^etre 
normes dans $K/k$.

\smallskip
On peut cependant donner un aper\c cu plus g\'en\'eral de la question et faire des 
statistiques sur l'ensemble des relations en consid\'erant un petit nombre de $\ell_j$, 
donn\'es a priori, en testant si un produit $\prod_j \ell_j^{e_j}$, $e_j \geq 0$,
est la norme d'un entier $x$ sans facteur rationnel de $k=\Q(\sqrt m)$, 
auquel cas on a la relation $\prod_j {\mathfrak l}_j^{e_j} = (x)$
entre les ${\mathfrak l}_j$ et on d\'etermine la
r\'epartition des $\delta_3(x)$ obtenus de cette fa\c con. 

\smallskip
Pour simplifier, on a consid\'er\'e un cas o\`u $h=3$ et o\`u $\delta_3(\varepsilon)=4$~;
la variable {\sf  Npx} compte le nombres de produits non principaux et 
{\sf  Nn} le nombre total de produits test\'es.
Les $e_j$ sont pris au hasard dans $\{0, 1, 2\}$, un grand nombre de fois~; 
il n'est pas n\'ecessaire de conna\^itre les ordres  des classes des ${\mathfrak l}_j$
car on obtient des statistiques tr\`es stables, quel que soit l'ensemble de
$\ell_j$ retenu. De fait le r\'esultat est relativement naturel dans la mesure
o\`u caract\'eriser $\delta_3(x)$ revient \`a faire des statistiques dans des groupes 
de classes g\'en\'eralis\'ees (i.e., modulo un rayon puissance de $p$)~; or les
th\'eor\`emes de densit\'e conduisent \`a des r\'epartitions canoniques.

\smallskip
On a consid\'er\'e des $\ell_j \equiv 1 \pmod 9$, mais les r\'esultats sont identiques pour
des $\ell_j \equiv 1 \pmod {3^{n_0}}$, $n_0$ arbitraire, sauf que les produits
$\prod_j \ell_j^{e_j}$ deviennent tr\`es grands ainsi que le temps de calcul~:

\footnotesize
\begin{verbatim}
==============================================================================
{p=3;m=7249;Q=x^2-m;K=bnfinit(Q,1);B=10^3;Mp=p^2;listL=List;NlistL=0;L=1;
while(L<B,L=L+2*Mp;if(isprime(L)==1 & kronecker(m,L)==1,NlistL=NlistL+1;
listinsert(listL,L,1)));C0=0;C1=0;C2=0;Nn=0;Npx=0;for(n=1,10^3,PL=1;
for(k=1,NlistL,PL=PL*component(listL,k)^random(3));N=bnfisintnorm(K,PL);
d=matsize(N);d1=component(d,1);d2=component(d,2);if(d2==0,Npx=Npx+1);
if(d2!=0,for(j=1,d2,aa=component(N,j);if(aa!=1,a1=component(aa,1);
a2=component(aa,2);if(gcd(a1,a2)==1,a=Mod(aa,Q);Nn=Nn+1;A=(a^2-1)/3;
v=valuation(A,3);if(v==0,C0=C0+1);if(v==1,C1=C1+1);if(v>=2,C2=C2+1))))));
print(Nn," ",Npx," ",Npx/Nn+0.0);print(" ");
print(C0/Nn+0.0," ",C1/Nn+0.0," ",C2/Nn+0.0);print(" ")}
==============================================================================
\end{verbatim}
\normalsize
Pour l'exemple de $m= 7249$, on obtient les donn\'ees suivantes (\S\,\ref{2sub1})~:
\footnotesize
\begin{verbatim}
m=7249, h=3, E=170524677024744665220*x+14518651659981320194199
p=3, Eta=131480821*x-11194416394
1   1/81
\end{verbatim}
\normalsize
pour lesquelles $\delta_{\mathfrak p}(\eta_3)=0$ et $\delta_3(\varepsilon)=4$~; on a test\'e
la liste suivante de nombres premiers $\ell \equiv 1 \pmod 9$, d\'ecompos\'es dans $k$~:

\smallskip
\centerline{${\sf listL}=[937,\  883,\  811,\  631,\  487,\  181,\  163,\  37]$.}

\smallskip
On obtient (o\`u ${\sf CJ}=\order \{x, \, \delta_3(x)= j\}$, $j=0,1$,
${\sf C2} = \order \{x, \, \delta_3(x)\geq 2\}$)~:
\footnotesize
\begin{verbatim}
Nn=11018  Npx=448  proportion=0.04066073
C0/Nn=0.66563804  C1/Nn=0.29606099  C2/Nn=0.03830096
\end{verbatim}
\normalsize
Plusieurs passages du programme donnent~:
${\sf Nn}=10720$, {\sf Npx}$=454$, et les proportions attendues de 
$x$ tels que $\delta_3(x)=0$~:

\centerline{$0.666001,\,  0.666652,\,  0.668400,\,  0.667622,\,  
0.667252,\,  0.666424,\,  0.666396$.}

\smallskip
Nous revenons \`a l'\'etude des facteurs classes et normique \`a partir des 
$\ell$-unit\'es du corps $k$ pour un grand nombre de $\ell$.

\subsection{\bf Facteurs classes et normique de $\ell$-unit\'es
pour $m=72262$}\label{2sub3}

On a $\varepsilon = 632566365854478210\cdot \sqrt m + 170043910956651732101$
et la $3$-unit\'e fondamentale $\eta_3 = -1037963 \cdot \sqrt m + 279020981$.

\smallskip
On a $\delta_{\mathfrak p}(\eta_3)=0$, $\delta_3(\varepsilon)=4$, 
$\Cl_k \simeq \Z/9\, \Z$ est engendr\'e par ${\mathfrak p} \mid 3$
(on a aussi $\Cl_{k_1} \simeq \Z/27\, \Z$ et $\Cl_{k_2} \simeq \Z/81\,\Z$).
La condition suffisante pour avoir $\lambda=\mu=0$ est donc satis\-faite et 
il est int\'eressant de voir si cela se traduit sur 
ces \'etudes de normes. 

\smallskip
Le module $3^{8+1}$, qui figure un calcul dans 
$k_8/k$ pour des $\ell$ totalement d\'ecompos\'es, peut \^etre modifi\'e 
\`a volont\'e car on observe que les statistiques n'en d\'ependent pas.

\subsubsection{Programme}
Le programme est le suivant o\`u $r$ est une puissance 
de $p=3$ divisant le nombre de classes de $k$ (ici $r \in\{1, 3, 9\}$). 
Il calcule, pour chaque premier $\ell$ 
totalement d\'ecompos\'e dans $K/\Q$ (class\'es par $\ell \equiv 1 \pmod {3^{n+1}}$, 
puis $-1\pmod {3^{n+1}}$), la ${\mathfrak l}$-unit\'e fondamentale $\eta_\ell$, 
o\`u ${\mathfrak l}$ est un id\'eal premier de $k$ au-dessus de $\ell$, dont la 
classe est d'ordre $r$ donn\'e et il calcule $\delta_3(\eta_\ell)$~; 
on aura donc $(\eta_\ell)={\mathfrak l}^{r}$ (cas particulier de
relations de principalit\'e). Ici  $n=8$.

\smallskip
La r\'epartition des ordres des classe des ${\mathfrak l}$ constitue la premi\`ere 
partie du programme et seulement en seconde partie, on utilise la 
valeur de $r$ fix\'ee au d\'ebut.

\smallskip
Le nombre ${\sf NLr}$ repr\'esente le nombre de premiers 
$\ell \leq {\sf BL} = 2\cdot 10^{12}$ 
totalement d\'ecompos\'es dans $K/k$ tels que la classe de 
${\mathfrak l}= {\rm N}_{K/k}({\mathfrak L})$ soit d'ordre~$r$.
Les proportions sont compar\'ees aux probabilit\'es naturelles
$\frac{1}{9}$, $\frac{2}{9}$, $\frac{6}{9}$. 

\smallskip
On a ${\sf NL=NL1+NL3+NL9}$ (nombre de $\ell$ consid\'er\'es).

\smallskip
Pour chaque $r$, on d\'esigne par ${\sf C0, C1, C2, C3, C4, C5}$ les 
nombres de $\ell \leq {\sf BL}$ tels que $(\eta_\ell)={\mathfrak l}^{r}$ et
$\delta_3(\eta_\ell) = 0, 1, 2, 3, 4, \geq 5$ respectivement~:

\footnotesize
\begin{verbatim}
==============================================================================
{r=3;p=3;m=72262;n=8;BL=2*10^12;M=p^(n+1);Q=x^2-m;K=bnfinit(Q,1);
C0=0;C1=0;C2=0;C3=0;C4=0;C5=0;CL1=0;CL3=0;CL9=0;NL=0;NLr=0;
for(t=-1,0,L=2*t+1;while(L<BL,L=L+2*M;if(isprime(L)==1 & kronecker(m,L)==1,
NL=NL+1;Su=bnfsunit(K,idealprimedec(K,L));F=component(component(Su,1),1);
Eta=Mod(F,Q);No=norm(Eta);vcl=valuation(No,L);if(vcl==1,CL1=CL1+1);
if(vcl==3,CL3=CL3+1);if(vcl==9,CL9=CL9+1);if(vcl==r,NLr=NLr+1;
A=F;B=(Mod(A,Q)^2-1)/3;v=valuation(B,3);if(v==0,C0=C0+1);
if(v==1,C1=C1+1);if(v==2,C2=C2+1);if(v==3,C3=C3+1);if(v==4,C4=C4+1);
if(v>=5,C5=C5+1)))));print("p=",p," m=",m," n=",n," BL=",BL);
print("NLr=",NLr," C0=",C0," C1=",C1," C2=",C2," C3=",C3," C4=",C4," C5=",C5);
print(C0/NLr+0.0," ",C1/NLr+0.0," ",C2/NLr+0.0," ",C3/NLr+0.0," ",
C4/NLr+0.0," ",C5/NLr+0.0);S=0.0;for(j=1,8,S=S+(p-1.0)/p^(5+j));
print(2./3," ",2./9," ",2./27," ",2./81," ",2./243," ",S);print(" ");
print("r=",r);print("NL=",NL," CL1=",CL1," CL3=",CL3," CL9=",CL9);
print(CL1/NL+0.0," ",CL3/NL+0.0," ",CL9/NL+0.0);print(1./9," ",2./9," ",6./9)}
==============================================================================
\end{verbatim}
\normalsize

\subsubsection{R\'epartition des ordres des classes pour $m=72262$}
On obtient ${\sf NL}=5584183$, 

\noindent
${\sf NL1}=620512, {\sf NL3}=1240284, {\sf NL9}=3723387$,
et le tableau des proportions~:

\footnotesize
$$\begin{array}{cccc}
& \hbox{proportions} &   & \hbox{probabilit\'es} \vspace{0.1cm}\\ 
& {\sf NL1/NL}= 0.1111195675
 &\hspace{1.0cm}   \frac{2}{3} & = 0.1111111111 \\
& {\sf NL3/NL}= 0.2221066179
&\hspace{1.0cm} \frac{2}{3^2} & = 0.2222222222 \\
& {\sf NL9/NL}= 0.6667738145
&\hspace{1.0cm} \frac{2}{3^3} & = 0.6666666666 \\
\end{array} $$
\normalsize

Il est clair que l'heuristique de r\'epartition uniforme est v\'erifi\'ee.
Ceci traduit le comportement du facteur classes
destin\'e \`a devenir rapidement trivial sous r\'eserve du caract\`ere
al\'eatoire des ${\mathfrak L}$ obtenus par l'algorithme (nous y
reviendrons au \S\,\ref{3sub2}).

\subsubsection{R\'epartition des $\delta_3(\eta_\ell)$ pour $m=72262$}
Il reste \`a voir la r\'epar\-tition des $\delta_3(\eta_\ell)$ 
selon la valeur de $r$~:

\medskip
 (i) D\'enombrement des $\cl_k({\mathfrak l})$ principales et calcul 
des $\delta_3(\eta_\ell)$.

\smallskip
On obtient ${\sf NL1} = 620512\ $ avec $\ {\sf C0} = 412997, 
{\sf C1} = 137911, {\sf C2} = 46166$, 

\noindent
${\sf C3} = 15833, {\sf C4} = 5090, {\sf C5} = 2515$,
et le tableau~:

\footnotesize
$$\begin{array}{cccc}
& \hbox{proportions} &   & \hbox{probabilit\'es} \vspace{0.1cm}\\ 
& {\sf C0/NL1}= 0.6655745577
 &\hspace{1.0cm}   \frac{2}{3} & = 0.6666666666 \\
& {\sf C1/NL1}= 0.2222535583
&\hspace{1.0cm} \frac{2}{3^2} & = 0.2222222222 \\
& {\sf C2/NL1}= 0.0743998504
&\hspace{1.0cm} \frac{2}{3^3} & = 0.0740740740 \\
& {\sf C3/NL1}= 0.0255160254
&\hspace{1.0cm} \frac{2}{3^4} & = 0.0246913580 \\
& {\sf C4/NL1}= 0.0082029034
&\hspace{1.0cm}  \frac{2}{3^5} & = 0.0082304526 \\
& {\sf C5/NL1}= 0.0040531045
&\ \   \sum_{j \geq 6}\frac{2}{3^j} & = 0.0041152264
\end{array} $$
\normalsize

 (ii) D\'enombrement des $\cl_k({\mathfrak l})$ d'ordre $r=3$ et calcul 
des $\delta_3(\eta_\ell)$.

\smallskip
On obtient
${\sf NL3 = C0} =1240284$, ${\sf C1 = C2 = C3 = C4 = C5} = 0$.

\medskip
 (iii) D\'enombrement des $\cl_k({\mathfrak l})$ d'ordre $r=9$ et calcul 
des $\delta_3(\eta_\ell)$.

\smallskip
On obtient aussi
${\sf NL9 = C0} = 3723387$, ${\sf C1 = C2 = C3 = C4 = C5} = 0$.

\smallskip
Ceci s'explique par le fait que le $3$-groupe des classes du \go corps miroir\gf
$k^* := \Q(\sqrt{-3 \cdot 72262})$ est d'ordre $3$ et que, d'apr\`es le th\'eor\`eme
de Scholz, puisque l'unit\'e $\varepsilon$ de $k$ est $3$-primaire, il ne peut y avoir 
d'autres \go pseudo-unit\'es\gf $3$-primaires dans $k$, c'est-\`a-dire d'\'el\'ements 
$a \in k^\times \setminus k^{\times 3}$, non unit\'es, tels que $(a) = {\mathfrak a}^3$ 
et $a \equiv \pm 1 \pmod 9$ dans $k$ (i.e., $\delta_3(a) \geq 1$).

\smallskip
Ainsi, dans les cas $r=3$ et $r=9$, les $\delta_3(\eta_\ell)$ sont 
n\'ecessairement nuls. Mais ceci d\'epend de l'arithm\'etique de $k$ 
et ne concerne que des relations tr\`es particuli\`eres. 
De toutes fa\c cons cela va dans le bon sens pour le facteur normique
car un tel $\eta_\ell$ donne un symbole normique d'ordre maximum.

\subsection{\bf Facteurs classes et normique de $\ell$-unit\'es
pour $m = 10942$}\label{2sub4} 

On a $\varepsilon = 110617476121372232880\cdot \sqrt m 
+ 11571032155720815417599$
et une $3$-unit\'e fondamentale $\eta_3= -890711 \cdot \sqrt m + 93171947$.

\smallskip
On a $\delta_{\mathfrak p}(\eta_3)=1$, $\delta_3(\varepsilon)=6$, $h=3$,
$\Cl_k \simeq \Z/3\,\Z$ et $\Cl_{k_1} \simeq \Z/3\,\Z \times \Z/3\,\Z$.

\smallskip
Il s'agit d'un cas o\`u la condition suffisante pour avoir 
$\lambda=\mu=0$ n'est pas satisfaite en raison de l'aspect normique, 
mais le groupe des classes de $k$ est engendr\'e par un id\'eal premier
au-dessus de $3$.

\smallskip
Le programme est similaire au pr\'ec\'edent avec $r \in\{1, 3 \}$, 
les nombres ${\sf NL1}$, ${\sf NL3}$ repr\'esentent
le nombre de premiers $\ell \leq {\sf BL} = 2\cdot 10^{12}$, 
totalement d\'ecompos\'es dans $K/k$, tels que la classe de 
${\mathfrak l}={\rm N}_{K/k}({\mathfrak L})$ soit d'ordre $r=1, 3$
respectivement.
Les proportions correspondantes sont compar\'ees aux probabilit\'es 
naturelles qui sont ici $\frac{1}{3}$, $\frac{2}{3}$.
Les variables ${\sf C0, C1, C2, C3, C4, C5}$ sont 
les nombres de premiers $\ell$ tels que 
$\delta_3(\eta_\ell) = 0, 1, 2, 3, 4, \geq 5$ respectivement et on a
${\sf NL= NL1+NL3}$.

\subsubsection{R\'epartition des ordres des classe pour $m=10942$}
Les donn\'ees sur la r\'epartition des classe sont les m\^emes pour les valeurs 
de $r \in \{1, 3\}$, \`a savoir
${\sf NL}= 5587470, {\sf NL1}= 1862666, {\sf NL3}= 3724804$,
et pour les proportions~:

\footnotesize
$$\begin{array}{cccc}
& \hbox{proportions} &   & \hbox{probabilit\'es} \vspace{0.1cm}\\ 
& {\sf NL1/NL}= 0.333364832
 &\hspace{1.0cm}   \frac{2}{3} &  =0.3333333333  \\
& {\sf NL3/NL}= 0.666635167
&\hspace{1.0cm} \frac{2}{3^2} & = 0.6666666666  \\
\end{array} $$
\normalsize

L'heuristique de r\'epartition uniforme est encore v\'erifi\'ee pour les classes.

\subsubsection{R\'epartition des $\delta_3(\eta_\ell)$ pour $m=10942$}
Il y a deux cas \`a exami\-ner~:

\medskip
 (i) D\'enombrement des $\cl_k({\mathfrak l})$ principales et calcul 
des $\delta_3(\eta_\ell)$.

\smallskip
On obtient
${\sf NL1} = 1862666$, ${\sf C0} = 1241224$, ${\sf C1} = 414270$, 
${\sf C2} = 138044$, ${\sf C3} = 45829$, ${\sf C4} = 15475$, ${\sf C5} = 7824$
et le tableau~:

\footnotesize
$$\begin{array}{cccc}
& \hbox{proportions} &   & \hbox{probabilit\'es} \vspace{0.1cm}\\ 
& {\sf C0/NL1}= 0.6663696014
 &\hspace{1.0cm}   \frac{2}{3} &  = 0.6666666666 \\
& {\sf C1/NL1}= 0.2224070230
&\hspace{1.0cm} \frac{2}{3^2} & = 0.2222222222 \\
& {\sf C2/NL1}= 0.0741109785
&\hspace{1.0cm} \frac{2}{3^3} & = 0.0740740740 \\
&  {\sf C3/NL1}= 0.0246039816
&\hspace{1.0cm} \frac{2}{3^4} & = 0.0246913580 \\
&  {\sf C4/NL1}= 0.0083079843
&\hspace{1.0cm}  \frac{2}{3^5} & = 0.0082304526 \\
&  {\sf C5/NL1}= 0.0042004309
&  \sum_{j \geq 6}\frac{2}{3^j} & = 0.0041152264
\end{array} $$
\normalsize

 (ii) D\'enombrement des $\cl_k({\mathfrak l})$ d'ordre $3$ et calcul des 
$\delta_3(\eta_\ell)$.

\smallskip
On obtient
${\sf NL3} = 3724804$, ${\sf C0} = 0$, ${\sf C1} = 2484070$, 
${\sf C2 }= 827554 $, 

\noindent
${\sf C3} = 275264$, ${\sf C4} = 91971 $, ${\sf C5} = 45945$
et le tableau~:

\footnotesize
$$\begin{array}{cccc}
& \hbox{proportions} &   & \hbox{probabilit\'es} \vspace{0.1cm}\\ 
& {\sf C0/NL3}= 0.0000000000
 &\hspace{1.0cm}   0 &  =0.0000000000 \\
 & {\sf C1/NL3}= 0.6668995200
&\hspace{1.0cm} \frac{2}{3} & = 0.6666666666 \\
& {\sf C2/NL3}= 0.2221738378
&\hspace{1.0cm} \frac{2}{3^2} & = 0.2222222222 \\
& {\sf C3/NL3}= 0.0739002642
&\hspace{1.0cm} \frac{2}{3^3} & = 0.0740740740\\
& {\sf C4/NL3}= 0.0246915005
&\hspace{1.0cm} \frac{2}{3^4} & = 0.0246913580 \\
& {\sf C5/NL3}= 0.0123348772
&\hspace{1.0cm}  \frac{2}{3^5} & = 0.0123450517 \\
\end{array} $$
\normalsize

\smallskip
Les densit\'es et probabilit\'es doivent \^etre d\'ecal\'ees en raison de 
l'impos\-sibilit\'e, lorsque la classe de ${\mathfrak l}$ est d'ordre $3$,
du cas $\delta_3(\eta_\ell)=0$ qui s'explique comme suit 
(noter que, ici encore, la relation
${\mathfrak l}^3=(\eta_\ell)$ n'est qu'un cas tr\`es particulier de relation,
et que la plupart des relations $\prod_j {\mathfrak l}_j^{e_j} = (x)$ peuvent 
conduire \`a $\delta_3(x)=0$)~:

\smallskip
Pour le corps miroir $k^* = \Q(\sqrt{-3\cdot 10942})$ le nombre de classes 
est $216 = 8 \cdot 27$ et le $3$-groupe de classes est isomorphe \`a 
$\Z/9\,\Z \times \Z/3\,\Z$, ce qui explique que dans $k$, on doit avoir 
deux pseudo-unit\'es ind\'epen\-dantes $3$-primaires~; 
l'une est toujours donn\'ee par l'unit\'e fondamentale puisque
$\delta_3(\varepsilon)=6$ et une autre
au moyen d'un ${\mathfrak l}$ convenable dont la classe est d'ordre $3$ et tel que
${\mathfrak l}^3 = (b)$, $\delta_3(b)\geq 1$. Plus pr\'ecis\'ement~:

\smallskip
Soit $L = \Q(\sqrt {10942}, \sqrt {-3})$ contenant  le corps miroir $k^*$~; 
comme $3$ est ramifi\'e dans $k^*/\Q$ et d\'ecompos\'e dans $k$, 
le $3$-rang de ${\mathcal T}_L$ est \'egal \`a~:
$${\rm rg}_3(\Cl_L^{S_L}) + \order S_L - 1 = {\rm rg}_3(\Cl_L^{S_L})+1$$
\cite[Proposition III.4.2.2]{Gra1},
o\`u $S_L$ est l'ensemble des $3$-places de $L$.

\smallskip
Or $\Cl_L^{S_L}$ est la somme directe de $\Cl_k^{S_k} = 1$
(car $S_k$ engendre $\Cl_{k}$) et de $\Cl_{k^*}$ (car $S_{k^*}$
est r\'eduit \`a une unique $3$-place de carr\'e $(3)$).

\smallskip
Donc ${\rm rg}_3({\mathcal T}_L)=3$ puis, comme 
${\mathcal T}_L = {\mathcal T}_k \plus  {\mathcal T}_{k^*}$
et comme ${\rm rg}_3({\mathcal T}_k)=1$ \cite[Corollary III.4.2.3]{Gra1}, on a
${\rm rg}_3({\mathcal T}_{k^*}) = 2 = {\rm rg}_3(\Cl_{k^*})$,
et toute extension cyclique $3$-ramifi\'ee de degr\'e $3$ de $k^*$
est n\'ecessaire\-ment contenue dans le corps de 
Hilbert de $k^*$ et est donc non ramifi\'ee~; d'o\`u le 
fait que pour tout ${\mathfrak l}$ dont la classe est d'ordre $3$ avec
${\mathfrak l}^3 = (b)$, n\'ecessairement $b$ est $3$-primaire
(i.e., $\delta_3(b) \geq 1$), ce qui explique que  
exceptionnellement ${\sf C0}=0$. 

\smallskip
Mais bien entendu, pour les ${\mathfrak l} = (b)$ principaux, 
on a vu que la propri\'et\'e de r\'eparti\-tion uniforme 
des $\delta_3(b)$ reste vraie en toute circonstance~; ainsi la condition
$\delta_3(\eta_\ell)=0$ suppose de plus ${\mathfrak l}$ principal 
(probabilit\'e $\frac{1}{3}$).

\subsection{\bf Exemples de corps $k$ avec $\Cl_k^{S_k} \ne 1$,
$\delta_3(\varepsilon) \geq1$ \& $\delta_{\mathfrak p}(\eta_3) \geq1$}\label{2sub5}
On a trouv\'e les cas suivants (avec $h=3$), pour lesquels aucun des deux points 
de la condition suffisante de nullit\'e de $\lambda$ et $\mu$ n'est v\'erifi\'e~;
on donne en outre la structure du groupe des classes de $k_1$~:

\medskip
 (i) $m=26893, \varepsilon=142445225/2 \cdot x + 23359714011/2$,

\qquad $\eta_3 = -x - 164$,

\qquad $\delta_3(\varepsilon)=3$ \& $\delta_{\mathfrak p}(\eta_3)=3$,
structure=[36, [18, 2]].

\medskip
 (ii) $m=31069, \varepsilon=933602804601721/2 \cdot x + 164560570852019805/2$,

\qquad $\eta_3 = -23257 \cdot x + 4099372$,

\qquad $\delta_3(\varepsilon)=3$ \& $\delta_{\mathfrak p}(\eta_3)=1$,
structure=[27, [3, 3, 3]].

\medskip
 (iii) $m=92269, \varepsilon=182039966136652680262184737485085/2 \cdot x $

\hfill $+ 55296119237149041291682191243785961/2$,

\qquad $\eta_3 = -3397805798209/2 \cdot x - 1032111126747851/2$,

\qquad $\delta_3(\varepsilon)=3$ \& $\delta_{\mathfrak p}(\eta_3)=1$,
structure=[9, [9]]

\medskip
 (iv) $m=94918, \varepsilon=188160617208817500397435811509434 \cdot x  $

\hfill $+ 57969962353214358861329455735908197$,

\qquad $\eta_3 =-8591 \cdot x - 2646781$,

\qquad $\delta_3(\varepsilon)=3$ \& $\delta_{\mathfrak p}(\eta_3)=2$,
structure=[27, [9, 3]]

\medskip
 (v) $m=171061, \varepsilon=900555961068792369443990032360047045/2 \cdot x $

\hfill$+ 372465634300948242809059190380968649273/2$,

\qquad $\eta_3 =902353/2 \cdot x - 373208881/2$,

\qquad $\delta_3(\varepsilon)=4$ \& $\delta_{\mathfrak p}(\eta_3)=4$,
structure=[9, [9]].

\subsubsection{Remarques sur les exemples pr\'ec\'edents}
(i) Pour $m = 31069$ et $r=3$, on trouve des r\'esultats analogues 
\`a ceux du second exemple o\`u ${\sf C0}=0$~:
${\sf NL3} = 3721754$, ${\sf C0}=0$, ${\sf C1}= 2480548$, ${\sf C2}= 826613$, 
${\sf C3}= 276142$, ${\sf C4}= 92283$, ${\sf C5}= 46168$.

\smallskip
(ii) Pour $m \in \{26893, 92269, 94918, 171061\}$ et $r=3$, on trouve
comme pour le premier exemple~:
${\sf C0 = NL3}$, ${\sf C1 = C2 = C3 = C4 = C5} = 0$.

\section{Equation d'\'evolution $(y) = {\mathfrak B}\!\cdot \! {\mathfrak A}^{1-\sigma}$
--  Obstruction $p$-adique}\label{sec3}

Cette section est consacr\'ee \`a  l'observation du passage de l'\'etape $i$
de l'algo\-rithme \`a l'\'etape $i+1$, c'est-\`a-dire \`a l'obtention de nouveaux id\'eaux
pour constituer ${\mathcal I}_{i+1}^n$ \`a partir de ${\mathcal I}_i^n$ (cf. Th\'eor\`eme 
\ref{filtration} \& \eqref{lambda})
et sur la question de savoir si ces id\'eaux et leurs normes
sont \go al\'eatoires\gf ou non par rapport aux pr\'ec\'edents.

\subsection{\bf Point fondamental de l'algorithme de calcul 
des $\order \Cl_{k_n}$}\label{3sub1}
Une fois le groupe $\Lambda_i^n$ d\'etermin\'e, l'\'etape suivante de l'algo\-rithme 
consiste \`a trouver les $x \in \Lambda_i^n$ normes locales en $p$, ce qui 
r\'esulte des valeurs des $\delta_p(x)$ (Relation \eqref{qf}). 
On pose alors, en vertu du th\'eor\`eme des normes de Hasse~: 
$$\hbox{$x = {\rm N}_{K/k}(y),\ \ y\in K^\times$, \ \hbox{d\'efini 
modulo $K^\times{}^{1-\sigma }$}}, $$
et comme un tel $x$ est par d\'efinition norme dans $K/k$ 
d'un id\'eal ${\mathfrak B} \in {\mathcal I}_i^n$, on a l'existence de 
${\mathfrak A}$ \'etranger \`a $p$, tel que~: 
$$(y) = {\mathfrak B}\cdot {\mathfrak A}^{1- \sigma}, $$
et le but est de v\'erifier, au moyen de statistiques num\'eriques, 
l'ind\'epen\-dance du nouveau pas $i+1$ de l'algorithme par rapport 
au pas $i$ pr\'ec\'edent (autrement dit, que ${\mathfrak A}$ n'a aucune relation 
{\it alg\'ebrique} avec ${\mathfrak B}$ et constitue un nouveau 
\go tirage probabiliste\gf\!\!\!). Comme ${\mathfrak A}$ peut \^etre 
d\'efini au produit pr\`es par $({\mathfrak a}) \cdot (z)$, o\`u 
$({\mathfrak a})$ est l'\'etendu d'un id\'eal de $k$ et $z \in K^\times$,
${\rm N}_{K/k}({\mathfrak A})$ est d\'efini au produit
pr\`es par ${\mathfrak a}^{p^n} \cdot {\rm N}_{K/k}(z)$ 
dont le symbole d'Artin dans $F/k$ est trivial pour $n \gg 0$, d'o\`u
l'aspect intrins\`eque du processus.
Ce point est l'\'el\'ement crucial de l'analyse heuristique 
de la conjecture de Greenberg.

\smallskip
Ensuite on s'int\'eresse \`a ${\mathfrak A}$ pour construire ${\mathcal I}_{i+1}^n$, 
puis on consid\`ere sa norme ${\rm N}_{K/k}({\mathfrak A}) \in 
{\rm N}_{K/k}({\mathcal I}_{i+1}^n)$ pour obtenir
$\Lambda_{i+1}^n \supseteq \Lambda_{i}^n$, afin de prendre 
les $x'  \in \Lambda_{i+1}^n$ tels que 
$x'$ soit la norme d'un $y' \in K^\times$, etc., sachant que sous reserve
des Hypoth\`eses (H), fin du \S\,\ref{sub2}, on aura statistiquement des 
$x \in \Lambda_{i+1}^n$ tels que $\delta_p(x) < \delta_p(\Lambda_{i}^n)$, 
ce qui fait d\'ecro\^itre le facteur normique tandis que les classes des 
composantes ${\mathfrak t}$ des ${\rm N}_{K/k}({\mathfrak A})$ font 
d\'ecro\^itre le facteur classes.

\subsubsection{Remarques sur~: \go entiers normes locales\gf 
vs \go normes d'entiers\gf}\label{entiers}
(i) Le th\'eor\`eme des normes
de Hasse pour les extensions cycliques est, au plan num\'erique, assez
probl\'ematique et influence fortement notre d\'emarche heuristique puisque
il n'existe (\`a notre connaissance) aucune formule permettant de passer 
du local au global~; autrement dit, \`a supposer que les solutions locales
$y_w$ dans les compl\'et\'es $K_{w}$ au-dessus de $v$ \`a l'\'equation 
normique $x = {\rm N}_{K_{w}/k_v}(y_w)$, soient connues
pour toute place $v$ de $k$ et $w \mid v$ de $K$, il n'est pas possible 
d'en d\'eduire une solution globale $y$.

\smallskip
(ii) La relation $x = {\rm N}_{K/k}(y)$, pour $x \in k^\times$ partout norme 
locale dans $K/k$ et {\it entier} (cas auquel on peut toujours se ramener), est 
assez subtile car $(x)$ est norme d'un id\'eal {\it entier} ${\mathfrak B}$, 
et s'il existe une solution $y$ {\it enti\`ere}, alors $(y)={\mathfrak B}$ 
est principal dans $K$~; 
dans le cas contraire, on peut \'ecrire $y = \Frac{z}{\Delta}$, 
$\Delta \in \Z$, $z$ {\it entier} de $K$, 
et il est facile de constater que l'id\'eal ${\mathfrak A}$ de la relation 
$(y) = {\mathfrak B}\cdot {\mathfrak A}^{1- \sigma}$ est 
essentiellement construit avec les id\'eaux premiers au-dessus 
de $\Delta $ dans $K$. Ce d\'enominateur d\'efinit de fait le caract\`ere 
al\'eatoire de l'algorithme.

\smallskip
(iii) Il faut signaler le cas trivial $x=1$ et l'\'equation ${\rm N}_{K/k}(y)=1$ 
qui conduit, par le Th\'eor\`eme\,$90$ de Hilbert, \`a $y = z^{1-\sigma}$, 
$z \in K^\times$, o\`u une solution $z_0$ est donn\'ee, pour une extension 
cyclique de degr\'e $N$, par une r\'esolvante de Hilbert de la forme~:
$$z_0= t+y \cdot t^\sigma + \cdots + 
y^{1+ \sigma +\cdots + \sigma^{i-1}}\!\! \cdot \ t^{\sigma^i}
+ \cdots + y^{1+ \sigma +\cdots + \sigma^{N-2}}\!\!\cdot \ t^{\sigma^{N-1}}$$
pour $t \in K^\times$ convenable rendant $z_0$ non nul~; 
mais l'aspect additif sugg\`ere justement une part totalement al\'eatoire.
Dans le cas de l'application du th\'eor\`eme des normes de Hasse et
de l'\'equation norme en id\'eaux qui en r\'esulte, il n'y a pas de r\'esolvante
explicite, mais on peut penser qu'il y a une complexit\'e du m\^eme ordre.

\begin{lemma} Sans modifier l'algorithme ni les statistiques, 
on peut sup\-poser que ${\mathfrak B}$ est un id\'eal premier 
${\mathfrak L}$ totalement d\'ecompos\'e dans~$K/\Q$. 
Dans la relation $(y) = {\mathfrak L}\cdot {\mathfrak A}^{1-\sigma}$ 
qui s'en d\'eduit, on peut supposer que ${\mathfrak A}$  est un 
id\'eal premier ${\mathfrak Q}$, totalement d\'ecompos\'e dans $K/\Q$.
\end{lemma}

\begin{proof}
Dans l'algorithme, on peut modifier tout \'el\'ement $x \in \Lambda_i^n$ modulo 
${\rm N}_{K/k}(K^\times)$, ce qui est \'equivalent 
\`a choisir ${\mathfrak B}$ modulo un id\'eal principal~; on peut donc 
supposer ${\mathfrak B}={\mathfrak L}$ premier, totalement d\'ecom\-pos\'e 
dans $K/\Q$.

\smallskip
Il est clair que $y$ peut \^etre modifi\'e modulo $K^{\times 1-\sigma}$ puisque 
seul ${\rm N}_{K/k}({\mathfrak A})$ est utilis\'e~; par cons\'equent,
${\mathfrak A}$ peut \^etre aussi d\'efini modulo un id\'eal principal, et
on peut supposer que ${\mathfrak A}={\mathfrak Q}$, premier
totalement d\'ecom\-pos\'e dans $K/\Q$.
Ceci \'equivaut au fait que ${\rm N}_{K/k}({\mathfrak Q}) =: {\mathfrak q}$ est 
un id\'eal premier au-dessus de $q \equiv \pm 1 \pmod {p^{n+1}}$ dont
la composante ${\mathfrak t}$ est inchang\'ee. $\square$
\end{proof}

\subsubsection{Remarques sur le calcul de PARI}
(i) Dans les calculs, PARI ne donne pas 
directement ${\mathfrak A}={\mathfrak Q}$, mais le plus souvent un id\'eal 
${\mathfrak A}$ de la forme ${\mathfrak Q}^\omega$, $\omega \in 
\Z[{\rm Gal}(K/\Q)]$ non n\'ecessairement \'egal \`a $1$, ce qui conduit \`a 
la relation~:
$${\rm N}_{K/k}({\mathfrak A}) = {\rm N}_{K/k}({\mathfrak Q})^{\omega(1)} 
=: {\mathfrak q}^{\omega(1)}, $$
o\`u ${\mathfrak q}$ est l'id\'eal premier de $k$ au-dessous de ${\mathfrak Q}$
et $\omega(1) \in\Z$ l'image de $\omega$ dans l'application d'augmentation.
On a donc \`a r\'esoudre, en un id\'eal premier ${\mathfrak Q}$~:
$$(y) = {\mathfrak L}\cdot {\mathfrak Q}^{\omega \cdot (1-\sigma)}, $$
o\`u ${\mathfrak L}$, id\'eal premier de $K$ totalement d\'ecompos\'e, est donn\'e 
tel que~:
\begin{equation}\label{equa7}
{\rm N}_{K/k}({\mathfrak L})={\mathfrak l}=:(x)=({\rm N}_{K/k}(y))
\end{equation}
(id\'eal premier principal de $k$ au-dessous de ${\mathfrak L}$).

\medskip
(ii) L'aspect statistique devra v\'erifier que 
${\mathfrak l}={\rm N}_{K/k}({\mathfrak L})$ et 
${\mathfrak q}^{\omega(1)}\!= {\rm N}_{K/k}({\mathfrak Q}^\omega)$ 
sont {\it  ind\'ependants} du point de
vue classes d'id\'eaux et propri\'et\'es $p$-adiques (au sens des sections 
pr\'ec\'edentes). Ces calculs seront programm\'es au \S\,\ref{3sub3}.

\subsection{\bf Remarques heuristiques fondamentales}\label{3sub2}
(i) On peut objecter que notre d\'emarche pose question 
en ce sens que l'on peut concevoir logiquement les deux 
\go implications\gf sui\-vantes~:

\smallskip
\quad (a) C'est l'ensemble des groupes de classes des $k_n$ 
(donc les valeurs de $\lambda$ et $\mu$) qui \go pr\'eexistent\gf et 
qui \go imposent\gf\!\!, pour chaque $n$, les algorithmes num\'e\-riques 
qui les d\'eterminent et en particulier qui imposent, quel que soit $n$,
le nombre de pas $m_n=O(1)\cdot (\lambda\cdot n + \mu\cdot p^n)$, 
non born\'e si $\lambda$ ou $\mu$ est non nul, alors que la complexit\'e
de l'algorithme ne d\'epend que du groupe fini ${\mathcal T}_k$. 

\smallskip
\quad (b) On peut au contraire, {\it en examinant la nature des calculs},
se convain\-cre du fait que ce sont bien ces calculs impr\'evisibles 
(alg\'ebriquement parlant) qui conditionnent les r\'esultats et dire que ce sont plut\^ot les 
algorithmes num\'eriques qui \go font exister\gf les groupes de classes pour chaque $n$, 
puis leur limite projective. 
Ces calculs sont fond\'es sur la r\'esolution, en l'inconnue ${\mathfrak A}$, 
de l'\'equation pr\'ec\'edente $(y) = {\mathfrak B}\cdot {\mathfrak A}^{1-\sigma}$ 
(cf. Remarques \ref{entiers})
et sur la d\'etermination de quotients de Fermat 
$\big(\frac{x^{p-1}-1}{p}\big) =: \prod_{{\mathfrak p} \in S_k}
{\mathfrak p}^{\delta_{\mathfrak p}(x)} \!\cdot {\mathfrak b}_{(x)}$, 
${\mathfrak b}_{(x)}$ \'etranger \`a~$p$, conduisant aux 
$\delta_{\mathfrak p}(x)$. 

\smallskip
A priori, il n'y a pas d'obstruction \`a une r\'epartition uniforme des 
composantes ${\mathfrak t} \in {\mathcal T}_k$ associ\'ees aux 
${\rm N}_{K/k}({\mathfrak A})$ en raison des propri\'et\'es du 
symbole d'Artin de ${\rm N}_{K/k}({\mathfrak A})$
dans $\Gamma_\infty^{p^n} \oplus {\mathcal T}_{k}$
(Th\'eor\`eme \ref{delta} et suite exacte \eqref{suite}) ou
vu dans ${\rm Gal}(F/k)$. Or ${\mathfrak t}$ 
g\`ere les deux facteurs (classes et normique) dans le cadre
habituel du corps de classes associ\'e aux th\'eor\`emes de densit\'e.
Par contre, la succession des id\'eaux de ${\mathcal I}_i^n$ \`a
${\mathcal I}_{i+1}^n$ n'est pas alg\'ebriquement pr\'evisible.

\smallskip
Pour l'aspect statistique sur la donn\'ee num\'erique du corps
$k$, il est impossible de s'affranchir du fait
que la structure arithm\'etique de $K/k$ commande les \'etapes 
de l'algorithme. D'o\`u la n\'ecessit\'e de consid\'erer des
familles de corps~$k$.

\smallskip
 (ii) Sur un plan math\'ematique, on peut penser que la notion de 
complexit\'e algorithmique de ce type de calculs (\'equations pr\'ec\'edentes et 
passage \`a la limite dans la tour) d\'epend de ph\'enom\`enes de 
transcendance (complexe et/ou $p$-adique) comme pour le cas de 
la conjecture de Leopoldt que l'on peut consid\'erer comme de
nature proche de celle du point de vue (b) ci-dessus sur la 
conjecture de Greenberg~:

\smallskip
En effet, pour la conjecture de Leopoldt, 
la complexit\'e repose sur les calculs, modulo $p^n$, des d\'eterminants 
des d\'eveloppements $p$-adiques des logarithmes d'une unit\'e de Minkowski 
du corps $k$ et de ses conjugu\'ees. Ici, les algorithmes sont li\'es par la condition 
(triviale) de r\'eduction modulo $p^n$ du calcul modulo $p^{n+h}$, $h \geq 0$
(voir \cite{Gra6} pour l'\'etude des r\'egulateurs $p$-adiques).

\smallskip
 (iii) Par ailleurs, l'influence de la complexit\'e arithm\'etique du corps 
de base $k$ et celle de $p$ (quant \`a son ordre de grandeur) sont 
manifestes comme le montrent les exemples suivants~:

\smallskip
\quad (a) Pour la conjecture de Leopoldt, on peut trouver des 
r\'egulateurs $p$-adiques arbitrairement proches de $0$ en 
prenant par exemple des corps quadratiques $k = \Q(\sqrt m)$
avec $m= a^2\cdot p^{2 \rho} + 1$, suppos\'e sans facteur carr\'e, auquel cas
l'unit\'e fondamentale de $k$ est $\varepsilon = a\cdot p^\rho + \sqrt m$ dont le 
logarithme $p$-adique est \'equivalent \`a $p^\rho$, et cependant la 
conjecture de Leopoldt est ici trivialement vraie. 
Par exemple, $\varepsilon= 3^{26}+\sqrt m$, o\`u
$m=2 \!\cdot\! 17 \!\cdot\! 193 \!\cdot\! 
1249 \!\cdot\! 13729 \!\cdot\! 475356961 \!\cdot\! 780464337846444296785447886881
=1+3^{52}$,
pour laquelle $\delta_3(\varepsilon)=25$).

\smallskip
\quad (b) Le cas de la conjecture de Greenberg est plus d\'elicat, mais les aspects
num\'eriques d\'ependent essentiellement du groupe de torsion ${\mathcal T}_k$ 
(qui conjugue $p$-groupe des classes $\Cl_k$ et r\'egulateur $p$-adique 
normalis\'e ${\mathcal R}_k = U_k^* / \overline E_k$)~;
en particulier, on sait \cite[Th\'eor\`emes 4.7, 4.8, 4.10]{Gra3} 
que l'exposant $p^e$ du groupe $U_k^* / \overline E_k$ est 
une premi\`ere mesure de la complexit\'e. Enfin, d'apr\`es \cite{J2}, 
cette complexit\'e est aussi mesur\'ee par le comportement du 
{\it  groupe des classes logarithmiques} de $k$ dans~$k_\infty$.

\smallskip
 (iv) Il faut ajouter que, comme pour la conjecture de Leopoldt, 
les algorithmes de d\'evissage relatifs \`a la conjecture de Greenberg 
ne sont pas ind\'ependants par rapport \`a $n$
(en effet, c'est la th\'eorie d'Iwasawa qui structure leur ensemble \`a
partir d'un rang fini, 
ainsi que la th\'eorie du corps de classes qui indique par exemple que, 
pour tout $h \geq 0$, ${\rm N}_{k_{n+h}/k_n}(\Cl_{k_{n+h}})=\Cl_{k_n}$). 

\smallskip
Ce lien est exprim\'e par le fait que, pour $i \geq 1$ fix\'e
et tout $h\geq 1$, modulo des normes globales convenables
dans les extensions $k_{n+h}/k$ (ce qui ne modifie pas 
les indices $(\Lambda_i^{n+h} :  \Lambda_i^{n+h} \cap 
{\rm N}_{k_{n+h}/k}(k_{n+h}^\times))$), 
on peut obtenir, \`a partir du sch\'ema de $g$-modules 
(cf. Th\'eor\`eme \ref{filtration} \& \eqref{lambda})~:
\begin{equation}\label{schema}
\begin{array}{ccccccccc}  
1  & \too & M_i^{n+h} & \toooo & M^{n+h} &  \stackrel{(1-\sigma_{n+h})^i}{\toooo} 
&  (M^{n+h})^{(1-\sigma_{n+h})^i} &  \too 1 \\   \vspace{-0.4cm}   \\
& &   \Big \downarrow& & \hspace{-1.4cm}  {\rm N}_{k_{n+h}/k_n} \Big \downarrow & 
&\hspace{-1.4cm}  {\rm N}_{k_{n+h}/k_n} \Big \downarrow &  \\  \vspace{-0.5cm}   \\
1  &  \too   & M_i^n  
& \toooo &  M^n & \stackrel{(1-\sigma_n)^i}{\toooo} & (M^n)^{(1-\sigma_n)^i} & \too 1 \\
&&&& \downarrow  && \downarrow & \\
&&&& 1  && 1 &
\end{array} 
\end{equation}
les relations d'inclusions suivantes \cite[\S\,7.1, 
Sch\'ema \& Relation (7.1)]{Gra3}~:
\begin{equation*}
\begin{aligned}
& {\rm N}_{k_{n+h}/k} ({\mathcal I}^{n+h}_i ) \subseteq \cdots 
\subseteq {\rm N}_{k_{n+1}/k} ({\mathcal I}^{n+1}_i )
\subseteq {\rm N}_{k_{n}/k} ({\mathcal I}^n_i) \\
& E_k \subseteq  \Lambda^{n+h}_i \subseteq \cdots 
\subseteq \Lambda^{n+1}_i \subseteq \Lambda^{n}_i,
\end{aligned}
\end{equation*}
qui font que, en particulier, le sous-groupe engendr\'e par les 
composantes ${\mathfrak t}$ des normes d'id\'eaux
\`a l'\'etape $i$ est localement constant lorsque $n$ cro\^it.
Ce r\'esultat provient des fl\`eches verticales (\`a gauche)~:

\smallskip
\centerline{${\rm N}_{k_{n+h}/k_n} : M_i^{n+h} \tooo M_i^n, $}

\smallskip\noindent
qui ne sont a priori ni injectives ni surjectives, et que c'est 
\`a ce niveau que se trouve l'obstruction fondamentale \`a 
une preuve {\it  alg\'ebrique} de la conjecture de Greenberg,
parfois outrepass\'ee dans la litt\'erature.
En effet, si la conjecture est vraie, pour tout $n$ assez grand, 
les fl\`eches pr\'ec\'edentes sont des isomorphismes pour tout $i$
puisqu'alors, les ${\rm N}_{k_{n+h}/k_n}$ sont des isomorphismes de 
$g$-modules.

\smallskip
Il y a donc un lien entre la complexit\'e des algorithmes pour chaque $n$ et les 
valeurs de $\lambda$ et $\mu$~; en raison de la nature des calculs on peut penser 
que de tels algorithmes ont une complexit\'e semblable pour 
tout corps $k$ et tout $n$, ceci \'etant renforc\'e par l'id\'ee que les donn\'ees 
num\'eriques essentielles se lisent dans le corps de base. 

\section{Programmation de l'\'equation d'\'evolution -- Exemples pour $p=3$}

\subsection{\bf Programme g\'en\'eral de recherche de ${\mathfrak A}$ tel que
$(y)={\mathfrak B}\cdot {\mathfrak A}^{1-\sigma}$}\label{3sub3}

Il est impossible de faire des statistiques avec $K$ arbitraire\-ment 
grand dans la tour. 
Nous allons cependant examiner, au moyen d'un programme, 
le principe g\'en\'eral d'obtention de ${\mathfrak A} = {\mathfrak Q}^\omega$ 
et ${\rm N}_{K/k}({\mathfrak A})={\rm N}_{K/k}({\mathfrak Q}^\omega) 
= {\mathfrak q}^{\omega(1)}$ \`a partir de la relation \eqref{equa7}
du \S\,\ref{3sub1}
${\rm N}_{K/k}(y) = {\rm N}_{K/k}({\mathfrak L}) = {\mathfrak l} =: (x)$,
o\`u ${\mathfrak l} \mid \ell$ est un id\'eal premier principal de $k$
et o\`u $x =: \eta_\ell$ est norme locale en $p$, donc norme globale 
dans l'extension $K/k$. Ici $\eta_\ell$ est donc la ${\mathfrak l}$-unit\'e
fondamentale (de valuation $1$), d\'etermin\'ee \`a une unit\'e pr\`es.

\smallskip
Nous nous limitons \`a prendre $K=k_1$, ce qui reste significatif car
l'algorithme de filtration de $\Cl_{k_1}$ reste non trivial~; en effet, si la formule de 
Chevalley donne un premier ordre de grandeur, rien n'exige que l'algorithme
soit born\'e de fa\c con effective. 

\smallskip
Le programme g\'en\'eral I, qui d\'etermine la solution $y \in k_1^\times$ 
et la factorisation de ${\mathfrak A} = {\mathfrak Q}^\omega$ 
dans $(y) = {\mathfrak L} \!\cdot \! {\mathfrak A}^{1-\sigma}$,
se d\'ecompose en deux sous-programmes~:

\smallskip
 (i) le premier consid\`ere un id\'eal premier principal ${\mathfrak l}=(\eta_\ell)$ de $k$
pour lequel $\eta_\ell$ est norme globale dans l'extension $k_1/k$~; on supposera
$\ell$ totalement d\'ecompos\'e dans $k_{n}/k$, $n \geq 1$ (ici $n=8$), 
donc, a fortiori, de la forme ${\rm N}_{k_1/k}({\mathfrak L})$ pour ${\mathfrak L} 
\mid {\mathfrak l}$ dans $k_1$. Les r\'esultats ne d\'ependent 
pas du choix de $n$, ce qui renforce les heuristiques.

\smallskip
 (ii) le second utilise les $\eta_\ell$ pr\'ec\'edents pour calculer
(au moyen de  la fonction ${\sf rnfisnorminit}$ de PARI dans l'extension relative $k_1/k$)
$y \in k_1^\times$ tel que $(y) = {\mathfrak L}\cdot {\mathfrak A}^{1-\sigma}$,
et donne la factorisation de l'id\'eal ${\mathfrak A}={\mathfrak Q}^\omega$ dont la norme
${\mathfrak q}^{\omega(1)}$ dans $k_1/k$ contribue au pas suivant de l'algorithme.

\subsubsection{Programme I~:  Id\'eaux 
${\mathfrak l}=(\eta_\ell)$ et  ${\mathfrak A}$ tels 
que $(y) = {\mathfrak L}\cdot {\mathfrak A}^{1-\sigma}$} \label{67}

On rappelle que $k=\Q(\sqrt m)$, pour $m\equiv 1 \pmod 3$.
La premi\`ere partie du programme
redonne les caract\'eristiques du corps $k$ (nombre de classes $h$,
unit\'e fondamentale $\varepsilon$, $S_k$-unit\'e fondamentale $\eta_3$,
$\delta_{\mathfrak p}(\eta_3)$, $\delta_3(\varepsilon) \geq 1$, sous la forme 
$\Frac{1}{3^{\delta_{\mathfrak p}(\eta_3)}}$, $\Frac{1}{3^{\delta_3(\varepsilon)}}$)~; 
il calcule une liste de nombres premiers 
$\ell$ totalement d\'ecompos\'es dans $k_{n}/\Q$ tels que les id\'eaux premiers 
${\mathfrak l} \mid \ell$ soient principaux de la forme $(\eta_\ell)$ o\`u 
$\delta_3(\eta_\ell) \geq 1$ afin que $\eta_\ell$ soit partout norme locale 
dans $k_1/k$, donc norme globale. 

\smallskip
On recherche les $\ell < {\sf B}$,
$\ell \equiv \pm 1 \pmod{{\sf Mp}=9}$ et ${\mathfrak l} = (\eta_\ell)$ dans $k$, 
o\`u $\eta_\ell$ est calcu\'e sous la forme ${\sf Mod(a*z+b, z^2-m)}$, mais
les $\delta_{\mathfrak p}(\eta_3)$ et $\delta_3(\varepsilon)$ n\'ecessitent 
des calculs modulo une puissance de $3$ suffisante.
Il fournit les r\'esultats suivants indiquant (ici pour $m=67$) 
que $\delta_{\mathfrak p}(\eta_3)=1$ et $\delta_3(\varepsilon)=2$~:

\footnotesize
\begin{verbatim}
PB-NORMIQUE
m=67,h=1,E=5967*z+48842
p=3,Eta=z+8
1/3  1/9

List([67,9,List([991,883,487,379,953,773,683,557,251]),
Mod(4*z-9,z^2-67),Mod(31*z-252,z^2-67),Mod(-32*z+261,z^2-67),
Mod(5*z-36,z^2-67),Mod(4*z+45,z^2-67),Mod(-122*z+999,z^2-67),
Mod(18*z-145,z^2-67),Mod(-14*z-117,z^2-67),Mod(-45*z+368,z^2-67)])
\end{verbatim}
\normalsize

La liste obtenue contient dans l'ordre, $m$, le nombre de $\ell$ 
trouv\'es, et enfin la liste des $\eta_\ell$ qui sera exploit\'ee par la
troisi\`eme partie du programme.
La factorisation de $(y)$ se fait en d\'efinissant $k_1$ au moyen 
du polyn\^ome irr\'eductible $R$ de $X=(\zeta_9+\zeta_9^{-1}) \cdot \sqrt m$,
o\`u $\zeta_9$ est une racine primitive $9$-i\`eme de l'unit\'e~:
$$R=x^6-6m\,x^4+ 9 m^2 x^2 - m^3.$$

On obtient (toujours avec $m=67$) la factorisation de $(y)$ 
en id\'eaux premiers, ce qui permet d'en d\'eduire
${\mathfrak A}$ et ${\rm N}_{k_1/k}({\mathfrak A})$:

\smallskip
\footnotesize
\begin{verbatim}
==============================================================================
{p=3;m=67;n=8;n0=n+1;B=10^5;y=z;Q=z^2-m;K=bnfinit(Q,1); 
h=component(component(bnrinit(K,1),5),1);
E=component(component(component(K,8),5),1);
Su=bnfsunit(K,idealprimedec(K,p));
pi1=component(component(Su,1),1);
pi2=component(pi1,2)*z-component(pi1,1);
Pi1=pi1^n0;Pi2=pi2^n0;Z=bezout(Pi1,Pi2);
U1=component(Z,1);U2=component(Z,2);
Pk=y^2-Mod(m,p^n0);Y=Mod(y,Pk);z=Y;A1=eval(U1);A2=eval(U2);
B1=eval(Pi1);B2=eval(Pi2);b1=eval(pi1);b2=eval(pi2);e=eval(E);
XPpi=Mod(A1*B1+A2*B2*b2,Pk);XPe=Mod(A1*B1+A2*B2*e,Pk);
hs=norm(Mod(pi1,Q));h0=valuation(hs,p);vh0=valuation(h0,p);delta=vh-vh0;
npi=norm(XPpi)^(p-1);ne=norm(XPe)^(p-1);
zpi=znorder(npi)/p^n;ze=znorder(ne)/p^n;
if(delta!=0,print("PB-CLASSES"));if(zpi+ze<1,print("PB-NORMIQUE"));
print("m=",m," h=",h," E=",E);print("p=",p," Eta=",pi1);print(zpi," ",ze);

Mp=p^2;Nlist=0;list=List;listL=List;
for(t=-1,0,L=2*t+1;while(L<B,L=L+2*Mp;
if(isprime(L)==1 & kronecker(m,L)==1,
SuL=bnfsunit(K,idealprimedec(K,L));F=component(component(SuL,1),1);
Eta=Mod(F,Q);No=norm(Eta);vcl=valuation(No,L);
if(vcl==1,A=(Mod(F,Q)^2-1)/3;v=valuation(A,3);
if(v>=1,Nlist=Nlist+1;listinsert(listL,L,1);listinsert(list,Eta,1))))));
listinsert(list,m,1);listinsert(list,Nlist,2);listinsert(list,listL,3);
print(list);

bnf=bnfinit(y^2-m);PK=polsubcyclo(9,3)+Mod(0,bnf.pol);
T=rnfisnorminit(bnf,PK,1);R=x^6-6*m*x^4+9*m^2*x^2-m^3;K=nfinit(R);
X=Mod(x,R);racm=m^2/(3*m*X-X^3);z=X^2/m-2;
for(j=1,Nlist,Z=component(list,j+3);ZZ=component(Z,2);
ZZ1=component(ZZ,1);ZZ2=component(ZZ,2);Z=Mod(ZZ1+ZZ2*y,y^2-m);
N=rnfisnorm(T,Z);nu=component(N,2);
if(nu==1,Y0=component(N,1);S=component(Y0,2);
S0=component(S,1);S1=component(S,2);S2=component(S,3);
if(S2==0,a1=0;a0=0);if(S2!=0,s2=component(S2,2);
a0=component(s2,1);a1=component(s2,2));
if(S1==0,b1=0;b0=0);if(S1!=0,s1=component(S1,2);
b0=component(s1,1);b1=component(s1,2));
if(S0==0,c1=0;c0=0);if(S0!=0,s0=component(S0,2);
c0=component(s0,1);c1=component(s0,2));
YY=(a1*racm+a0)*z^2+(b1*racm+b0)*z+c1*racm+c0;F=idealfactor(K,YY);
L=component(listL,j);print(" ");print(L);print(F) ));z=y}
==============================================================================
\end{verbatim}
\normalsize
La factorisation indique chaque
id\'eal premier avec la donn\'ee d'une $\Z$-base et, \`a l'extr\'emit\'e droite,
l'exposant de l'id\'eal~; la norme de l'id\'eal est le premier entier \`a gauche~;
l'id\'eal ${\mathfrak L}$, qui figure par hypoth\`ese, est list\'e en dernier. 
On obtient par exemple la d\'ecomposition triviale (i.e., ${\mathfrak A}=1$)~:

\footnotesize
\begin{verbatim}
Mat([[991,[145,0,0,0,1,0]~,1,1,[385,-256,182,-61,334,-466]~],1]) 
\end{verbatim}
\normalsize
qui d\'ecrit un id\'eal principal ${\mathfrak L}$ de $k_1$ au-dessus de 
$\ell = 991$. Ensuite on a, par exemple, une factorisation de la forme~:

\footnotesize
\begin{verbatim}
[[181,[-55,0,0,0,1,0]~,1,1,[-84,37,11,-9,88,-36]~],1; 
[181,[-7,0,0,0,1,0]~,1,1,[-34,-45,-71,61,-63,-36]~],-1; 
[487,[-110,0,0,0,1,0]~,1,1,[196,51,-28,226,-135,106]~],1] 
\end{verbatim}
\normalsize
qui d\'ecrit le produit de l'id\'eal premier ${\mathfrak L}$ au-dessus de 
$487$ par ${\mathfrak Q}^{1-\sigma}$, o\`u ${\mathfrak Q}$ est un
id\'eal premier au-dessus d'un id\'eal ${\mathfrak q}$ de $k$
divisant $q= 181$ choisi par PARI. 

\medskip
En r\'esum\'e, on obtient, pour $m=67$, les factorisations relatives \`a la
liste des $\ell$ pr\'ec\'edente~:

\footnotesize
\begin{verbatim}
l=991:
Mat([[991,[145,0,0,0,1,0]~,1,1,[385,-256,182,-61,334,-466]~],1])
l=883:
Mat([[883,[-395,0,0,0,1,0]~,1,1,[67,287,91,-236,-207,74]~],1])
l=487:
[[181,[-55,0,0,0,1,0]~,1,1,[-84,37,11,-9,88,-36]~],1;
[181,[-7,0,0,0,1,0]~,1,1,[-34,-45,-71,61,-63,-36]~],-1;
[487,[-110,0,0,0,1,0]~,1,1,[196,51,-28,226,-135,106]~],1]
l=379:
[[181,[-55,0,0,0,1,0]~,1,1,[-84,37,11,-9,88,-36]~],1;
[181,[-7,0,0,0,1,0]~,1,1,[-34,-45,-71,61,-63,-36]~],-1;
[379,[129,0,0,0,1,0]~,1,1,[-141,1,31,-12,-137,-59]~],1]
l=953:
[[181,[-55,0,0,0,1,0]~,1,1,[-84,37,11,-9,88,-36]~],2;
[181,[-7,0,0,0,1,0]~,1,1,[-34,-45,-71,61,-63,-36]~],-2;
[953,[-53,0,0,0,1,0]~,1,1,[202,374,-333,-215,-115,-276]~],1]
l=773:
Mat([[773,[-145,0,0,0,1,0]~,1,1,[-14,277,39,-355,343,-149]~],1])
l=683:
[[181,[-55,0,0,0,1,0]~,1,1,[-84,37,11,-9,88,-36]~],1;
[181,[-7,0,0,0,1,0]~,1,1,[-34,-45,-71,61,-63,-36]~],-1;
[683,[-240,0,0,0,1,0]~,1,1,[339,41,269,-141,-283,-292]~],1]
l=557:
[[181,[-55,0,0,0,1,0]~,1,1,[-84,37,11,-9,88,-36]~],1;
[181,[-7,0,0,0,1,0]~,1,1,[-34,-45,-71,61,-63,-36]~],-1;
[557,[-30,0,0,0,1,0]~,1,1,[-67,112,-124,111,-277,33]~],1]
l=251:
[[181,[-55,0,0,0,1,0]~,1,1,[-84,37,11,-9,88,-36]~],1;
[181,[-7,0,0,0,1,0]~,1,1,[-34,-45,-71,61,-63,-36]~],-1;
[251,[5,0,0,0,1,0]~,1,1,[-63,120,-106,-53,89,-29]~],1]
\end{verbatim} 
\normalsize

On constate que PARI utilise bien le m\^eme id\'eal premier ${\mathfrak Q} \mid 181$, 
ici avec $\omega = f$, $f \in \{0, 1, 2\}$, de sorte que ${\mathfrak A}$ est \'egal \`a 
${\mathfrak Q}^0=(1)$ ou \`a ${\mathfrak Q}$ ou \`a ${\mathfrak Q}^2$.
Les 3 cas se pr\'esentent effectivement.

\subsubsection{Programme II~: Id\'eaux ${\mathfrak A}$ tels
que $(y) = {\mathfrak A}^{1-\sigma}\ \,  \& \ \,  {\rm N}_{k_1/k}(y)= \varepsilon$}

C'est un cas particulier du programme pr\'ec\'edent qui permet de trouver
les classes ambiges, autres que celles des id\'eaux invariants, et les
$\delta_3(x)$, $x \in \Lambda_1^1$, tr\`es importants pour la suite de l'algorithme.
On r\'esoud l'\'equation ${\rm N}_{k_1/k}(y)= \varepsilon$
qui a toujours une solution dans la mesure o\`u l'on suppose 
$\delta_3(\varepsilon) \geq 1$~; on donne aussi la structure de $\Cl_{k_1}$~:

\footnotesize
\begin{verbatim}
============================================================================
{m=67;bnf=bnfinit(y^2-m);E=component(component(component(bnf,8),5),1);
E=Mod(E,y^2-m);PK=polsubcyclo(9,3)+Mod(0,bnf.pol);
T=rnfisnorminit(bnf,PK,1);R=x^6-6*m*x^4+9*m^2*x^2-m^3;K=nfinit(R);
X=Mod(x,R);racm=m^2/(3*m*X-X^3);z=X^2/m-2;
N=rnfisnorm(T,E);nu=component(N,2);if(nu==1,Y=component(N,1);
S=component(Y,2);S0=component(S,1);S1=component(S,2);
S2=component(S,3);if(S2==0,a1=0;a0=0);if(S2!=0,s2=component(S2,2);
a0=component(s2,1);a1=component(s2,2));if(S1==0,b1=0;b0=0);
if(S1!=0,s1=component(S1,2);b0=component(s1,1);b1=component(s1,2));
if(S0==0,c1=0;c0=0);if(S0!=0,s0=component(S0,2);
c0=component(s0,1);c1=component(s0,2));
YY=(a1*racm+a0)*z^2+(b1*racm+b0)*z+c1*racm+c0;F=idealfactor(K,YY);print(Y);
print(F));H=bnrinit(bnfinit(R,1),1);print("structure=",component(H,5))}
============================================================================
\end{verbatim}
\normalsize

\smallskip
(i) Pour $m=67$ ($h=1, \delta_{\mathfrak p}(\eta_3)=1, \delta_3(\varepsilon)=2$,
$\Cl_{k_1} \simeq \Z/3 \Z$), $y$ est~:

\footnotesize
\begin{verbatim}
Mod(Mod(872/181*y+7113/181,y^2-67)*x^2+Mod(104/181*y+850/181,y^2-67)*x
+Mod(-2313/181*y-18873/181,y^2-67),x^3-3*x+Mod(1,y^2-67))
\end{verbatim}
\normalsize

L'id\'eal ${\mathfrak A}$ est l'id\'eal premier~:

\footnotesize
\begin{verbatim} 
[181,[7,0,0,0,1,0]~,1,1,[34,45,71,61,-63,-36]~]
\end{verbatim}
\normalsize

On v\'erifie que
$181 = {\rm N}_{k/\Q}(3\sqrt {67} + 28)$ pour lequel $\delta_3(3\sqrt {67} + 28)=0$.
On a $M_1^1=3$ d'ordre $3$, engendr\'e par $S_k$ et ${\mathfrak A}$~; or 
${\rm N}_{k_1/k}(M_1^1)=1$ et $\Lambda_1^1 = \langle \varepsilon, 
\eta_3, 3\sqrt {67}+28 \rangle$ qui stope l'algorithme.

\medskip
(ii) Pour $m=6559$ ($h=18, \delta_{\mathfrak p}(\eta_3)=1, \delta_3(\varepsilon)=3$,
$\Cl_{k_1}\simeq \Z/27 \Z \times \Z/3 \Z$), $y$ est~:

\footnotesize
\begin{verbatim}
Mod(Mod(2603286587676/74632321*y-210708142484324/74632321,y^2-6559)*x^2
+Mod(904557609272/74632321*y-73082414174026/74632321,y^2-6559)*x
+Mod(-7493532286573/74632321*y+606807427105320/74632321,y^2-6559),
                                                      x^3-3*x+Mod(1,y^2-6559))
\end{verbatim}
\normalsize
et l'id\'eal ${\mathfrak A}^{1-\sigma}$ est le produit~:

\footnotesize
\begin{verbatim}
[[53,[-9,0,0,0,1,0]~,1,1,[25,-11,-16,-3,-22,10]~],-2;
[53,[7,0,0,0,1,0]~,1,1,[21,-12,-17,-11,-24,10]~],2;
[163,[-49,0,0,0,1,0]~,1,1,[16,-52,38,23,-34,54]~],-2;
[163,[-41,0,0,0,1,0]~,1,1,[-35,22,-68,-21,-52,54]~],1;
[163,[-8,0,0,0,1,0]~,1,1,[-5,33,-57,-3,-78,54]~],-1;
[163,[8,0,0,0,1,0]~,1,1,[5,-33,57,-3,-78,54]~],2]
\end{verbatim}
\normalsize
pour lequel il faudra d\'eterminer les conjugaisons pour trouver $\omega$.

\smallskip
(iii) Pour $m= 3259$ ($h=1, \delta_{\mathfrak p}(\eta_3)=0, \delta_3(\varepsilon)=4$,
$\Cl_{k_1}\simeq \Z/3 \Z$), 
on trouve ${\mathfrak A}=(1)$ qui montre que $\varepsilon$ est norme de l'unit\'e de $k_1$~:

\footnotesize
\begin{verbatim}
Mod(Mod(-495452848877794109154272*y-28284239771961302173706384,y^2-3259)*x^2
+Mod(931133218427718936425952*y+53156209050568241456130896,y^2-3259)*x
+Mod(-263604189149463218625499*y-15048544190803267053864318,y^2-3259),
                                                      x^3-3*x+Mod(1,y^2-3259))
\end{verbatim}
\normalsize
et par cons\'equent, les classes ambiges sont les classes des id\'eaux inva\-riants,
donc de $S_k$ puisque $h=1$. Or $\order M_1^1=3$, et comme 
$\Lambda_1^1 = \langle \varepsilon, \eta_3 \rangle$ avec $\delta_{\mathfrak p}(\eta_3)=0$,
l'algorithme stope, ce qui est coh\'erent avec $\order \Cl_{k_1}=3$.

\smallskip
 (iv)  Pour $m=1867$  ($h=1, \delta_{\mathfrak p}(\eta_3)=1, \delta_3(\varepsilon)=5$,
$\Cl_{k_1}\simeq \Z/3 \Z$),
on obtient respectivement pour $y$ et ${\mathfrak A}$~:

\footnotesize
\begin{verbatim}
Mod(Mod(488982/107*y+21128286/107,y^2-1867)*x^2
+Mod(-778797/107*y-33650892/107,y^2-1867)*x
+Mod(442398/107*y+19115590/107,y^2-1867),x^3-3*x+Mod(1,y^2-1867))

[107,[40,0,0,0,1,0]~,1,1,[-16,36,-33,24,53,-50]~],1
\end{verbatim}
\normalsize
o\`u ${\rm N}_{k/\Q} (34086\,\sqrt m + 1472815)=-107$ et
$\delta_3(34086\,\sqrt m + 1472815)=0$.

\medskip
Certains de ces exemples donnent au premier stade $\Lambda_1^1$ des 
$\delta_3(x)$ qui stopent l'algorithme ($m=67, 3259, 1867$ pour lesquels
$\Cl_{k_1} \simeq \Z/3\Z$).

\subsection{\bf Evolution de la $i$-suite $\order (M_{i+1}^1/M_i^1)$ 
pour $k=\Q(\sqrt {6559})$}

Ce cas est particuli\`erement int\'eressant en raison du groupe de classes cyclique 
d'ordre $9$ et du r\'egulateur $3$-adique normalis\'e $\order {\mathcal R}_k={27}$, 
ce qui donne un groupe ${\mathcal T}_k$ d'ordre $3^5$, donc une grande 
vari\'et\'e de d\'ecompositions~; en outre, on a 
$\Cl_{k_1} \simeq \Z/27\Z \times \Z/3\Z$.

\subsubsection{Utilisation du programme I, \S\,\ref{67}}\label{3sub4}
Il donne les r\'esultats suivants pour un ensemble de premiers $\ell$ totalement 
d\'ecompos\'es dans $k_1/\Q$ et tels que ${\mathfrak l} = (\eta_\ell)$ dans $k$~:

\footnotesize
\begin{verbatim}
PB-NORMIQUE
m=6559,h=18,E=81*z+6560
p=3,Eta=379*z-30694
1/3  1/27
\end{verbatim}
\normalsize
indiquant que $\delta_{\mathfrak p}(\eta_3)=1$ et
$\delta_3(\varepsilon) = 3$. Noter que
$\Cl_k$ est engendr\'e par ${\mathfrak p}\mid 3$.

\smallskip
On obtient la liste suivante des $\eta_\ell$ dans la 
troisi\`eme composante de la variable ${\sf list}$~:

\footnotesize
\begin{verbatim}     
List([6559,25,
List([92179,86239,70327,68743,58321,48907,47143,40519,30781, 
28279,25237,12547,8011,1621,96461,96263,89009,88001,84653, 
82457,77003,37781,27179,16361,2267]), 

Mod(-70*z+5661,z^2-6559),Mod(65*z+5256,z^2-6559),Mod(-52*z+4203,z^2-6559),
Mod(16*z+1269,z^2-6559),Mod(-135*z-10936,z^2-6559),Mod(7*z-522,z^2-6559), 
Mod(11*z-864,z^2-6559),Mod(29*z+2340,z^2-6559),Mod(-90*z-7291,z^2-6559),
Mod(20*z+1611,z^2-6559),Mod(9*z-746,z^2-6559),Mod(-2*z-117,z^2-6559), 
Mod(-2*z-135,z^2-6559),Mod(-9*z-730,z^2-6559),Mod(61*z-4950,z^2-6559),
Mod(36*z+2899,z^2-6559),Mod(-20*z-1647,z^2-6559),Mod(-56*z-4545,z^2-6559),
Mod(-2*z-333,z^2-6559),Mod(11*z+936,z^2-6559),Mod(9*z+674,z^2-6559),
Mod(25*z-2034,z^2-6559),Mod(9*z+710,z^2-6559),Mod(-11*z-900,z^2-6559),
Mod(18*z+1457,z^2-6559)])
\end{verbatim}
\normalsize

\subsubsection{D\'ecomposition de $(y)$ en id\'eaux, pour $k=\Q(\sqrt {6559})$}
Le Programme I du \S\,\ref{67} donne la solution $y\in k_1^\times$ telle que
${\rm N}_{k_1/ k}(y) = \eta_\ell$ et la d\'ecompo\-sition en id\'eaux de $(y)$~;
le nombre $y \in k_1^\times$ (identifi\'e par ${\sf Y}$ dans le programme) est 
d\'ecrit par PARI en termes polynomiaux (variables $x$ et $y$, modulo 
$y^2-m$ pour $k/\Q$ et modulo $x^3-3\,x+1$ pour l'extension $k_1/k$).

\smallskip
Par exemple, pour $\ell = 86239$ et $\eta_\ell=65\cdot \sqrt m + 5256$ 
de norme $\ell$, on a~:

\footnotesize
\begin{verbatim}
Y=Mod(Mod(103603429803986698500793761812835866687
                     8738304822327197272934756399840529/418195493*y
-8390598661945059443075872989882636906451249
                     9733687980568899041217527878770/418195493,y^2-6559)*x^2 
+Mod(-194710756949834302518150983617741998570
                     9413036014952120204634371585899310/418195493*y 
+15769167293211778079551042284522907830683899
                     9479469060631851001865336037874/418195493,y^2-6559)*x 
+Mod(551262335748355458007199134755867939
                     764504856477125221099310399055725544/418195493*y 
-4464544296904031158210117579530451263319333
                  4530869826508158991263404771360/418195493,y^2-6559),
                                                      x^3-3*x+Mod(1,y^2-6559))
\end{verbatim}
\normalsize

On aura en g\'en\'eral
$(y) = {\mathfrak L}\cdot {\mathfrak Q}^{\omega \cdot (1-\sigma)} , \ \,
\omega \in \Z[{\rm Gal}(k_1/\Q)]$,
o\`u ${\rm Gal}(k_1/\Q)$ est engendr\'e par $\sigma$ d'ordre $3$ 
et $\tau$ d'ordre $2$. 
Le programme donne de fait le produit $\Omega := \omega \cdot (1-\sigma)$
effectu\'e. 

\smallskip
On obtient des r\'esultats montrant le caract\`ere al\'eatoire 
de $\Omega$, dont les coefficients sont dans l'intervalle $[-9, +9]$
et o\`u l'id\'eal ${\mathfrak A}={\mathfrak Q}^\Omega$ est tel que 
${\mathfrak Q} \mid 53$ o\`u ${\rm N}_{k_1/k}({\mathfrak Q}) = 
{\mathfrak q}$ dont la classe est d'ordre $9$ (on rappelle que
les $\ell \equiv \pm 1 \pmod{81}$ sont class\'es selon les deux 
congruences, par ordres d\'ecroissants, et que les 
${\mathfrak l}$-unit\'es $\eta_\ell$ sont donn\'ees au \S\,\ref{3sub4}).
On observe en particulier 2 cas d'id\'eaux 
${\mathfrak Q}^\omega$ triviaux (donc lorsque l'id\'eal
${\mathfrak L}$ de $k_1$ est principal \'egal \`a $y$)~:

\footnotesize
\begin{verbatim}
1621
Mat([[1621,[-62,0,0,0,1,0]~,1,1,[119,-637,235,621,776,762]~],1])

27179
Mat([[27179,[-5996,0,0,0,1,0]~,1,1,[10104,8864,-2451,5717,-5400,-3876]~],1])
\end{verbatim}
\normalsize

Mais trouver $\omega$ n\'ecessite d'identifier
num\'eriquement les conjugu\'es de ${\mathfrak Q}$.

\subsubsection{D\'etermination de ${\rm Gal}(k_1/\Q)$ pour $k=\Q(\sqrt{6559})$}
On obtient, relative\-ment au polyn\^ome $R=x^6 - 39354\,x^4 + 387184329\,x^2 - 282171334879$,
et en utilisant ${\sf nfgaloisconj}$~:
\begin{equation*}
\begin{aligned}
\tau \hbox{\ \ d\'efini par\ \  }  x & \mapsto -x, \\
\sigma \hbox{\ \ d\'efini par\ \  }  x & \mapsto -1/43020481\,x^5 + 5/6559\,x^3 - 6\,x, \\
\tau \cdot \sigma^2 \hbox{\ \ d\'efini par\ \ } x & \mapsto -1/43020481\,x^5 + 5/6559\,x^3 - 5\,x, \\
\sigma^2 \hbox{\ \ d\'efini par\ \  } x & \mapsto\ \  1/43020481\,x^5 - 5/6559\,x^3 + 5\,x, \\
\tau \cdot \sigma \hbox{\ \ d\'efini par\ \  } x & \mapsto\ \  1/43020481\,x^5 - 5/6559\,x^3 + 6\,x.
\end{aligned}
\end{equation*}

On a alors calcul\'e, dans l'ordre $1, \tau, \sigma, \tau \cdot \sigma^2, 
\sigma^2, \tau \cdot \sigma$,
les conjugu\'es corres\-pondants de $y$ et sa d\'ecomposition en id\'eaux,
ce qui permet d'identifier les conjugu\'es de l'id\'eal ${\mathfrak L}$ et 
de trouver $\Omega$~; on se base une fois pour toutes sur le $y$ 
obtenu pour $\ell=28279$~:

\footnotesize
\begin{verbatim}
Id(y):
[[53,[-7,0,0,0,1,0]~,1,1,[-21,12,17,-11,-24,10]~],3;      Q1
[53,[-2,0,0,0,1,0]~,1,1,[-22,15,20,-9,23,10]~],4;         Q2
[53,[2,0,0,0,1,0]~,1,1,[22,-15,-20,-9,23,10]~],-2;    tau(Q2)
[53,[7,0,0,0,1,0]~,1,1,[21,-12,-17,-11,-24,10]~],2;   tau(Q1)
[53,[9,0,0,0,1,0]~,1,1,[-25,11,16,-3,-22,10]~],-7;        Q3
[28279,[3506,0,0,0,1,0]~,1,1,[356,7859,1221,-1277,3345,8122]~],1]

tau(y):
[[53,[-9,0,0,0,1,0]~,1,1,[25,-11,-16,-3,-22,10]~],-7;  tau(Q3)
[53,[-7,0,0,0,1,0]~,1,1,[-21,12,17,-11,-24,10]~],2;        Q1
[53,[-2,0,0,0,1,0]~,1,1,[-22,15,20,-9,23,10]~],-2;         Q2
[53,[2,0,0,0,1,0]~,1,1,[22,-15,-20,-9,23,10]~],4;      tau(Q2)
[53,[7,0,0,0,1,0]~,1,1,[21,-12,-17,-11,-24,10]~],3;    tau(Q1)
[28279,[-3506,0,0,0,1,0]~,1,1,[-356,-7859,-1221,-1277,3345,8122]~],1]

sigma(y):
[[53,[-9,0,0,0,1,0]~,1,1,[25,-11,-16,-3,-22,10]~],-2;  tau(Q3)
[53,[-7,0,0,0,1,0]~,1,1,[-21,12,17,-11,-24,10]~],-7;       Q1
[53,[-2,0,0,0,1,0]~,1,1,[-22,15,20,-9,23,10]~],3;          Q2
[53,[2,0,0,0,1,0]~,1,1,[22,-15,-20,-9,23,10]~],2;      tau(Q2)
[53,[9,0,0,0,1,0]~,1,1,[-25,11,16,-3,-22,10]~],4;          Q3
[28279,[-7370,0,0,0,1,0]~,1,1,[5417,-865,-7503,12899,3500,8122]~],1]

tau.sigma^2(y):
[[53,[-9,0,0,0,1,0]~,1,1,[25,-11,-16,-3,-22,10]~],3;  tau(Q3)
[53,[-7,0,0,0,1,0]~,1,1,[-21,12,17,-11,-24,10]~],-2;      Q1
[53,[2,0,0,0,1,0]~,1,1,[22,-15,-20,-9,23,10]~],-7;    tau(Q2)
[53,[7,0,0,0,1,0]~,1,1,[21,-12,-17,-11,-24,10]~],4;   tau(Q1)
[53,[9,0,0,0,1,0]~,1,1,[-25,11,16,-3,-22,10]~],2;         Q3
[28279,[-3864,0,0,0,1,0]~,1,1,[14138,-12920,-6282,12744,-10758,8122]~],1]

sigma^2(y):
[[53,[-9,0,0,0,1,0]~,1,1,[25,-11,-16,-3,-22,10]~],2;  tau(Q3)
[53,[-7,0,0,0,1,0]~,1,1,[-21,12,17,-11,-24,10]~],4;       Q1
[53,[-2,0,0,0,1,0]~,1,1,[-22,15,20,-9,23,10]~],-7;        Q2
[53,[7,0,0,0,1,0]~,1,1,[21,-12,-17,-11,-24,10]~],-2;  tau(Q1)
[53,[9,0,0,0,1,0]~,1,1,[-25,11,16,-3,-22,10]~],3;         Q3
[28279,[3864,0,0,0,1,0]~,1,1,[-14138,12920,6282,12744,-10758,8122]~],1]

tau.sigma(y):
[[53,[-9,0,0,0,1,0]~,1,1,[25,-11,-16,-3,-22,10]~],4;  tau(Q3)
[53,[-2,0,0,0,1,0]~,1,1,[-22,15,20,-9,23,10]~],2;         Q2
[53,[2,0,0,0,1,0]~,1,1,[22,-15,-20,-9,23,10]~],3;     tau(Q2)
[53,[7,0,0,0,1,0]~,1,1,[21,-12,-17,-11,-24,10]~],-7;  tau(Q1)
[53,[9,0,0,0,1,0]~,1,1,[-25,11,16,-3,-22,10]~],-2;        Q3
[28279,[7370,0,0,0,1,0]~,1,1,[-5417,865,7503,12899,3500,8122]~],1]
\end{verbatim}
\normalsize

Ceci identifie les relations de conjugaison \`a partir de ${\mathfrak Q}_1$~:
$${\mathfrak Q}_1,\ \ \ \ {\mathfrak Q}_1^\sigma = {\mathfrak Q}_2, \ \ \ \ 
{\mathfrak Q}_1^{\sigma^2} = {\mathfrak Q}_3. $$
D'o\`u, pour le cas de $\ell = 28279$,
$\Omega = 3+2\tau + (4-2\tau) \sigma -7 \sigma^2$, puis
$\omega = 7 \sigma +2\tau +3$ et $\omega(1)= 10+2\tau$.
Mais comme ${\mathfrak q}^{1+\tau} = (53)$, on obtiendra
${\rm N}_{k_1/k}({\mathfrak A} )= (53)^2 \cdot {\mathfrak q}^8$~;
or la classe de ${\mathfrak q}$ est d'ordre $9$, avec~: 
$${\mathfrak q}^9=(416167 \sqrt {6559} + 66601422) \ \ \& \ \ 
\delta_3(416167\sqrt {6559}  + 66601422)=1. $$

Pour $\ell = 82457$, on obtient~:

\footnotesize
\begin{verbatim}
l=82457:
[[53,[-7,0,0,0,1,0]~,1,1,[-21,12,17,-11,-24,10]~],2;   Q1
[53,[9,0,0,0,1,0]~,1,1,[-25,11,16,-3,-22,10]~],-2;     Q3
[82457,[-11462,0,0,0,1,0]~,1,1,[-12743,-34335,-22650,2405,35275,-22073]~],1]
\end{verbatim}
\normalsize
qui conduit \`a $\omega(1) = 4$.

\medskip
Pour $\ell = 2267$, il vient~:

\footnotesize
\begin{verbatim}
l=2267:
[[53,[-7,0,0,0,1,0]~,1,1,[-21,12,17,-11,-24,10]~],3;    Q1
[53,[-2,0,0,0,1,0]~,1,1,[-22,15,20,-9,23,10]~],1;       Q2
[53,[2,0, 0,0,1,0]~,1,1,[22,-15,-20,-9,23,10]~],-1; tau(Q2)
[53,[7,0,0,0,1,0]~,1,1,[21,-12,-17,-11,-24,10]~],1; tau(Q1)
[53,[9,0,0,0,1, 0]~,1,1,[-25,11,16,-3,-22,10]~],-4;     Q3
[2267,[855,0,0,0,1,0]~,1,1,[433,150,-991,-60,-759,-378]~],1]
\end{verbatim}
\normalsize
d'o\`u $\omega = 4 \sigma +3+\tau$, $\omega(1)=7+\tau$, et
${\rm N}_{k_1/k}({\mathfrak A}) = (53) \cdot {\mathfrak q}^6$, 
qui donne une classe d'ordre $3$.

\smallskip
Pour $\ell = 8011$ on obtient~:

\footnotesize
\begin{verbatim}
l=8011:
[[53,[-7,0,0,0,1,0]~,1,1,[-21,12,17,-11,-24,10]~],4;    Q1
[53,[-2,0,0,0,1,0]~,1,1,[-22,15,20,-9,23,10]~],3;       Q2
[53,[2,0, 0,0,1,0]~,1,1,[22,-15,-20,-9,23,10]~],-2; tau(Q2)
[53,[7,0,0,0,1,0]~,1,1,[21,-12,-17,-11,-24,10]~],2; tau(Q1)
[53,[9,0,0,0,1, 0]~,1,1,[-25,11,16,-3,-22,10]~],-7;     Q3
[8011,[-2063,0,0,0,1,0]~,1,1,[-2011,-1995,2900,979,-3590,1411]~],1]
\end{verbatim}
\normalsize
pour lequel
$\omega = 7 \sigma +4+2\tau$, $\omega(1)=11+2\tau = 9+2(1+\tau)$ qui conduit \`a 
l'id\'eal principal ${\rm N}_{k_1/k}({\mathfrak A}) = (53)^2 \cdot {\mathfrak q}^9$.

\smallskip
Par cons\'equent, tous les cas int\'eressants de $\omega$ sont obtenus, 
ce qui sugg\`ere la r\'epartition uniforme de la composante 
$\cl_k({\mathfrak t})$ dans $\Cl_k$ relativement \`a la d\'ecomposition 
de ${\rm N}_{k_1/k}({\mathfrak A})$ (Th\'eor\`eme \ref{delta}).

\subsubsection{Evolution du facteur normique pour $\Q(\sqrt{6559})$}\label{3sub5}
Le $3$-groupe des classes du corps $k^*=\Q(\sqrt{-3 \cdot 6559})$
est $\Cl_{k^*} \simeq (\Z/3\,\Z)^2$, ce qui fait que deux
pseudo-unit\'es $3$-primaires ind\'ependantes dans $k$ sont n\'ecessaires~;
ceci explique que, outre l'unit\'e fondamentale, tout $a$ \'etranger \`a $3$, 
tel que $(a)={\mathfrak a}^3$, v\'erifie n\'ecessairement
$\delta_3(a) \geq 1$ (cas analogue au cas de $m=10942$ du \S\,\ref{2sub4}).

\smallskip
Consid\'erons le groupe des classes ambiges $M_1^1$ d'ordre $27$~; 
il est engendr\'e par les classes de ${\mathfrak P} \mid 3$ dans $k_1$
et de ${\mathfrak A}={\mathfrak Q}_{53}^{-2} \cdot 
{\mathfrak Q}_{163}^{-2+(\tau-2)\cdot \sigma}$ provenant de la 
r\'esolution de $\varepsilon = {\rm N}_{k_1/k}(y)$ (Programme II), apr\`es 
identification des conjugu\'es et utilisation de ${\mathfrak A}^{1-\sigma}$ 
donn\'e par~:

\footnotesize
\begin{verbatim}
[[53,[-9,0,0,0,1,0]~,1,1,[25,-11,-16,-3,-22,10]~],-2;
[53,[7,0,0,0,1,0]~,1,1,[21,-12,-17,-11,-24,10]~],2;
[163,[-49,0,0,0,1,0]~,1,1,[16,-52,38,23,-34,54]~],-2;
[163,[-41,0,0,0,1,0]~,1,1,[-35,22,-68,-21,-52,54]~],1;
[163,[-8,0,0,0,1,0]~,1,1,[-5,33,-57,-3,-78,54]~],-1;
[163,[8,0,0,0,1,0]~,1,1,[5,-33,57,-3,-78,54]~],2]
\end{verbatim}
\normalsize

On remarque que $(\sqrt {6559}+80) = {\mathfrak p} \cdot {\mathfrak q}_{53}$~;
par cons\'equent ${\mathfrak q}_{53}$ et ${\mathfrak q}_{163}$ sont \'equivalents
 \`a une puissance de ${\mathfrak p}$. 
Donc ${\rm N}_{k_1/k}(M_1^1)$ est engendr\'e par les classes dans $k$ de
${\mathfrak p} \mid 3$ (car $\cl_k({\mathfrak p})$ engendre
 $\Cl_k$), de ${\mathfrak q}_{53}$ et ${\mathfrak q}_{163}$ (rajout\'es
par commodit\'e). On trouve avec PARI un $y \in k_1$ qui
a pour norme relative le nombre entier $x \in k$ suivant~:

\footnotesize
\begin{verbatim}
Mod(6832355788476479176909088393511957025*y 
                      -131997425842264293218558754198040661024,y^2-6559)
\end{verbatim}
\normalsize
de norme $-3^3 \cdot 53^{21} \cdot 163^{18}$, dont la d\'ecomposition 
en id\'eaux est~:

\footnotesize
\begin{verbatim}
[[3,[1,1]~,1,1,[-1,1]~]3]         P3
[[53,[26,1]~,1,1,[-26,1]~]21]     Q53
[[163,[56,1]~,1,1,[-56,1]~]18]    Q163
\end{verbatim}
\normalsize
et qui fournit une \'el\'ement de $\Lambda_1^1$ norme dans $k_1/k$
(le symbole de Hasse de $x$, non \'etranger \`a $3$, est de calcul plus 
complexe~; $y$ pourrait permettre le pas suivant de l'algorithme).
Donc $\order (M_2^1/M_1^1)=3$ car le facteur classes a \'et\'e 
trivialis\'e puisque ${\rm N}_{k_1/k}(M_1^1)=\Cl_k$. D'o\`u $\order M_2^1=81$, 
et comme PARI donne $\Cl_{k_1} \simeq \Z/27\Z \times \Z/3\Z$,
on a fin de l'algorithme (normalement on devrait d\'eterminer 
$M_3^1/M_2^1$ \`a partir du calcul de $\Lambda_2^1$ pour constater la fin).

\subsubsection{Remarques sur l'algorithme pour $\Q(\sqrt {6559})$}
(i) On v\'erifie que pour 
$k=\Q(\sqrt {6559})$ le $3$-groupe des classes de $k_2$
est isomorphe \`a $\Z/27\,\Z \times \Z/9\,\Z$. Le temps de calcul devient
important et il semble illusoire d'effectuer les calculs pr\'ec\'edents
dans $k_2/k$.

\smallskip
 (ii) On rappelle, d'apr\`es \cite[Th\'eor\`eme 4.7]{Gra3}, que l'exposant 
$3^e$ de $U_k^*/\overline  E_k$ indique l'\'etage (ici \'egal \`a $e=3$) 
\`a partir duquel le nombre de classes ambiges dans $k_n/k$ 
est \'egal \`a $\order {\mathcal T}_k = 3^5$ pour tout $n \geq e$. 
Nous ignorons si la stabilisation s'effectue \`a l'\'etage $3$, mais il 
est normal que le $3$-groupe des classes croisse au moins 
jusqu'\`a l'ordre $3^5$. 

\smallskip
Au-del\`a de cette borne, on aura 
$\order M_2^n = 3^{e_1^n} \cdot \order M_1^n = 3^{e_1^n} \cdot 3^5$,
o\`u $e_1^n \in \{0, 1\}$ ne d\'epend que du facteur normique selon les
modalit\'es abord\'ees au niveau $n=1$.

\smallskip
 (iii) Comme pour tous les $h \geq 0$, les normes ${\rm N}_{k_{n+h}/k_n} : 
\Cl_{k_{n+h}} \too \Cl_{k_n}$ sont surjectives, $\order \Cl_{k_n}$
est fonction croissante de $n$, mais avec la contrainte
$\order M_1^n = \order \Cl_{k_n}^{{\rm Gal}(k_n/k)} = 3^5$ pour tout $n \geq 3$
et le fait, rappel\'e \`a la fin de la Section \ref{sec1}, que la
$i$-suite des $\order (M_{i+1}^n/M_i^n) =: 3^{c_i^n+ \rho_i^n}$
est d\'ecroissante \`a partir de $\order M_1^n = 3^5$,
stationnaire, de limite un diviseur de $\order {\mathcal T}_k$.

\smallskip
Pour $m=6559$ o\`u ${\rm N}_{k_n/k}(S_{k_n})=S_k$
engendre $\Cl_k$, le facteur classes est tou\-jours trivialis\'e et tout d\'epend 
de la $i$-suite d\'ecroissante $3^{\rho_i^n} \!\mid\! (U_k^* : \overline  E_k) = 3^3$.

\section{Descente galoisienne de ${\mathcal T}_k$ via ${\rm Gal}(F/k)$}

Bien que la descente galoisienne de $H_k^{\rm pr}/k_\infty$, en $F/k$,
ne soit pas n\'ecessaire au plan th\'eorique, donnons d'abord un exemple 
num\'erique montrant le caract\`ere \go fini explicite\gf
des conditions de r\'epartition des symboles d'Artin 
$\Big( \Frac{F/k} {{\rm N}_{k_n/k}(\mathfrak A)} \Big)$ des normes
${\rm N}_{k_n/k}(\mathfrak A)$ dans la tour cyclotomique. 
Ensuite nous montrerons que cette r\'epartition des symboles d'Artin
ne d\'epend pas du choix de $F$ qui, en un sens, d\'efinit un 
\go corps gouvernant\gf pour la conjecture de Greenberg.

\subsection{\bf L'extension $F/k$ pour $k=\Q(\sqrt{1714})$, $p=3$}\label{3sub6}
Pour $m= 1714$ et $p=3$, le programme du \S\,\ref{prat} permet de
montrer que ${\mathcal T}_k \simeq \Z/3\Z \times \Z/3\Z$ avec $\Cl_k \simeq \Z/3\Z$~; 
le corps $k^*=\Q(\sqrt{-3\cdot m})$
a un $3$-groupe des classes isomorphe \`a $\Z/3\Z \times \Z/3\Z$ et  
il existe $\alpha, \beta \in k^*$, cubes d'id\'eaux non
principaux, tels que $\alpha$ soit par exemple $3$-primaire pour
engendrer l'extension de Kummer non ramifi\'ee 
$H_k(\mu_3^{}) = k(\mu_3^{})(\sqrt[3]{\alpha})$,
$\beta$ donnant une extension cyclique de degr\'e $3$ ramifi\'ee en $3$. 

\smallskip
Rappelons que si $(\gamma) = {\mathfrak c}^3$ dans $k^*$, 
avec ${\rm N}_{k^*/\Q}(\gamma) = N_\gamma^3$ et 
${\rm Tr}_{k^*/\Q}(\gamma) = T_\gamma$, alors l'extension cyclique de 
degr\'e $3$ de $k$ qui lui correspond est donn\'ee par le polyn\^ome~:
$$P_{\gamma} = x^3 - 3\,N_\gamma \!\cdot\! x - T_\gamma. $$
On obtient les donn\'ees suivantes~:
\begin{equation*}
\begin{aligned}
\alpha & = 2593 + 15 \cdot \sqrt{-5142},   &\hbox{avec $N_\alpha=199$}, \\
\beta & = 157 + \sqrt{-5142},   & \hbox{avec $N_\beta=31$},\ \,
\end{aligned}
\end{equation*}
$P_\alpha=x^3 - 3 \cdot 199 \, x - 2 \cdot  2593$ 
(non ramification sur $k$), $P_\beta=x^3 - 3 \cdot 31\, x - 2 \cdot 157$ 
(ramification en 3)~; en utilisant~:
$${\sf polcompositum(x^3-3*199*x-2*2593, x^3-3*31*x-2*157)}, $$
on obtient que $F$ est engendr\'ee sur $k$ par une racine de~:

\footnotesize
\begin{verbatim}
x^9-2070*x^7+14616*x^6+1261737*x^5-17516520*x^4-136713960*x^3
                                      +3712697856*x^2-22102948224*x+40749585408
\end{verbatim}
\normalsize

\noindent
pour lequel on v\'erifie que $D_{F/\Q}=D_k^9 \cdot 3^{27}$.
Ainsi le symbole d'Artin d'un id\'eal premier ${\mathfrak L}$ 
(par exemple totalement d\'ecompos\'e) de $k_n$,
obtenu au cours de l'algorithme, se lit sur la d\'ecomposition, dans $F/k$,
de l'id\'eal ${\mathfrak l}$ de $k$ au-dessous de ${\mathfrak L}$. Plus
g\'en\'eralement les symboles des ${\rm N}_{k_n/k}(\mathfrak A)$
en r\'esultent.

\subsection{\bf Invariance par rapport au choix de $F/k$}
L'extension $F/k$ n'est pas unique mais on a le r\'esultat suivant
qui conforte ces questions d'ordre heuristique et num\'erique~:

\begin{theorem} L'\'etude statistique des symboles d'Artin 
$\big( \frac{F/k} {{\rm N}_{k_n/k}({\mathfrak A})} \big)$, o\`u les 
id\'e\-aux ${\mathfrak A}$ sont obtenus dans l'algorithme de d\'evissage
dans $k_n$, est intrins\`eque pour tout $n$ assez grand et ne d\'epend 
pas du choix de $F$.
\end{theorem}

\begin{proof}
Soit $F'/k$ une autre solution~; alors en se r\'ef\'erant aux 
expressions du Th\'eor\`eme \ref{delta}, il vient, avec des 
notations \'evidentes pour $F$ et $F'$,
${\rm N}_{K/k}({\mathfrak A})= {\mathfrak a}^{p^n} \cdot
{\mathfrak t} \cdot (x_\infty) = {\mathfrak a}'{}^{p^n} \cdot
{\mathfrak t}' \cdot (x'_\infty)$,
ce qui conduit, d\`es que $n$ est assez grand, \`a
${\mathfrak t}' \cdot {\mathfrak t}^{-1} = (z)$, o\`u 
l'image de $z$ dans $U_k$
est arbitrairement proche de $1$. Comme ${\mathfrak t}$
et ${\mathfrak t}'$ sont d'ordre fini modulo ${\mathcal P}_{k,\infty}$,
on obtient, pour $e' \geq 0$ convenable,
${\mathfrak t}'{}^{p^{e'}} \cdot {\mathfrak t}^{-p^{e'}} 
= (z^{p^{e'}})=(t_\infty) \in {\mathcal P}_{k,\infty}$.
Donc $z^{p^{e'}} = t_\infty \cdot \varepsilon$, 
$\varepsilon \in E_k \otimes \Z_p$
d'image arbitrairement proche de $1$ dans $U_k$, 
donc de la forme $\varepsilon'{}^{p^{e'}}$,
$\varepsilon' \in E_k \otimes \Z_p$ (conjecture de Leopoldt), 
ce qui fait que $z' := z \cdot \varepsilon'{}^{-1}$
est tel que $z'{}^{p^{e'}} = t_\infty$. L'image de
$z'$ dans $U_k$ est dans ${\rm tor}_{\Z_p}(U_k)=1$ et $z'$ est infinit\'esimal
(cf. \S\,\ref{sub1}). D'o\`u $(z) = (z') \in {\mathcal P}_{k,\infty}$
et ${\mathfrak t}' \cdot {\mathfrak t}^{-1} \in {\mathcal P}_{k,\infty}$. $\square$
\end{proof}

\begin{remark} {\rm
Soit $p^{e'}$, ${e'} \geq e \geq 0$, l'exposant de ${\mathcal T}_k$
(o\`u l'on rappelle que $p^e$ est l'exposant de $U_k^*/\overline  E_k$), 
et pour tout $n \geq 0$, soit $F_n =k_nF= KF$ le sous-corps de $H_k^{\rm pr}$ 
fix\'e par $\Gamma_\infty^{p^n}$ (se reporter au sch\'ema du \S\,\ref{sub1}).
Alors, pour tout $n \geq {e'}$, la restriction ${\mathcal T}_k \too 
{\rm Gal}(KF/ K)$ est un isomorphisme de $g$-modules. En effet,
${\mathcal A}_k$ est normal dans ${\rm Gal}(H_k^{\rm pr}/\Q)$, et 
${\mathcal A}_k^{p^n} = \Gamma_\infty^{p^n}$ est normal et
fixe $KF$ qui est galoisien sur $\Q$.
Autremant dit, $g$ et ${\rm Gal}(K/\Q)$ op\`erent par conjugaison 
sur ${\mathcal T}_k$ de fa\c con coh\'erente. 

\smallskip
Ainsi, si $F$ (ou $\Gamma_\infty$) n'est pas unique, $F_{e'}$ est canonique 
comme sous-corps de $H_k^{\rm pr}$ fixe par ${\mathcal A}_k^{p^{e'}}$,
ce qui rend canonique, pour tout $n \geq e'$, la d\'ecomposition en 
id\'eaux ${\rm N}_{K/k}({\mathfrak A}) = {\mathfrak a}^{p^n} \!\cdot\! 
{\mathfrak t} \cdot (x_\infty)$, ${\mathfrak a}, {\mathfrak t} \in {\mathcal J}_k$,
et pr\'ecise le th\'eor\`eme pr\'ec\'edent.}
\end{remark}

\section{Conclusion}\label{concl}

Pour $k$ et $p$ fix\'es, les exp\'erimentations sugg\`erent que, pour tout 
$n \gg 0$ fix\'e, les probabilit\'es de trivialit\'e de chacun des deux facteurs 
de la $i$-suite $\order (M_{i+1}^n/M_i^n)$, pour $i$ croissant, 
{\it tendent rapidement vers $1$, ind\'ependamment de $n$}, 
selon des lois binomiales sur les pas successifs, $1 \leq i \leq m_n$.
Une estimation pr\'ecise est difficile en raison de la pr\'esence
de plusieurs param\`etres de type corps de classes comme certains 
exemples l'ont montr\'e aux \S\S\,\ref{2sub3}, \ref{2sub4}.

\smallskip 
Ceci rend cr\'edible l'hypoth\`ese et les heuristiques que nous avions donn\'ees dans
\cite[Hypoth\`ese 7.9, Heuristiques 7.5, 7.6]{Gra3} dont nous rappelons l'essentiel
pour un corps $k$ de degr\'e $d$, totalement r\'eel et $p$-d\'ecompos\'e~:

\medskip
 (i) {\it Soit $c \in \Cl_k$~; la probabilit\'e que, pour un id\'eal 
${\mathfrak A}$ de $k_n$ \'etranger \`a~$p$, la $p$-classe de 
${\rm N}_{k_n/k}({\mathfrak A})$ soit \'egale \`a $c$, est} $\Frac{1}{\order \Cl_k}$.
{\it La probabilit\'e que, pour $x \in \Lambda_i^n$, on ait 
$\delta_{\mathfrak p}(x) \geq r$, pour tout ${\mathfrak p} \in S_k$, est}~:
$\Frac{1}{p^{\,r \, (d-1)}}$~; {\it d'o\`u celle de}
$x \in  {\rm N}_{k_n/k}(k_n^\times)$.

(ii) {\it  Il existe $i_0\gg0$, ind\'ependant de $n$, tel que} l'on ait
${\rm N}_{k_n/k}(M^n_{i_0}) = \Cl_k$ et 
$\Frac{p^{n \cdot (d -1)}}{(\Lambda_{i_0}^n : \Lambda_{i_0}^n 
\cap {\rm N}_{k_n/k}(k_n^\times))} = 1$, pour tout $n \gg 0$.

\medskip
Ce que l'on peut r\'esumer de la fa\c con approximative suivante 
(d'autant plus que les probabilit\'es pr\'ec\'edentes sont 
conjecturalement des majorants)~:

\smallskip
{\it Pour $\lambda$ ou $\mu$ non nuls, la probabilit\'e de
$\order \Cl_{k_n} = p^{\lambda \cdot n + \mu \cdot p^n+\nu }$ est au plus en}
$\hbox{$\Frac{1}{p^{\,O(1)\,\cdot\, (\lambda \cdot n + \mu \cdot p^n)}}$,
pour tout $n \gg 0$.}$

\smallskip
L'existence de $i_0 \gg 0$, {\it ind\'ependant de $n \to \infty$}, stopant
les algorithmes, peut para\^itre arbitraire car 
les liens num\'eriques entre les \'etages $n$ et $n+h$ semblent 
difficilement analysables, \`a l'exeption du sch\'ema \eqref{schema}
du \S\,\ref{3sub2} \,(iv) qui tient compte \`a la fois 
des pas $i$ et des \'etages $n+h$ pour tout $h \geq 0$~; 
mais cette existence est 
renforc\'ee par le fait que $k$ v\'erifie la conjecture de Greenberg si et 
seulement si $\widetilde {\Cl_k}$ capitule dans $k_\infty$ (cf. \S\,\ref{log}).
En effet, si $\widetilde {\Cl_k}$ capitule dans $k_n$, il capitule 
dans $k_{n+h}$ pour tout $h\geq 0$. 
Il y a proba\-blement un lien concret entre le $k_{n_0}$ de 
capitulation de $\widetilde {\Cl_k}$ et $i_0$ qui pourrait \^etre li\'e au 
nombre de pas $m_{n_0}$ correspondant. De plus, une capitulation 
est progressive de $k$ \`a $k_{n_0}$, ce qu'il serait utile d'interpr\'eter
en termes d'algorithme de d\'evissage des $M_i^n$.

\smallskip
On pourrait traiter le cas d'un unique id\'eal premier 
dans $k$ au-dessus de $p$ car alors 
le facteur normique est toujours trivial et l'algorithme ne porte que 
sur le facteur classes. Quant au cas d'une d\'ecomposition partielle, 
il est clair que les m\^emes heuristiques s'appliquent, le r\'esultat 
de Jaulent \'etant g\'en\'eral et les formules \ref{cr} pouvant \^etre 
modifi\'ees en cons\'equence selon \cite{Gra2}. 

\smallskip
En conclusion, le comportement des $\Cl_{k_n}$ dans la tour ne 
d\'epend pas uniquement de {\it circonstances alg\'ebriques 
\`a la Iwasawa}, ni m\^eme de la th\'eorie du corps de classes
ou de celle des fonctions $L_p$, mais d'autres ph\'enom\`enes 
arithm\'etiques $p$-adiques subtils qui se lisent de fa\c con 
probabiliste au moyen des invariants habituels du 
corps $k$, sauf que la stabilisation pr\'ecise de $\order \Cl_{k_n}$, 
lorsque $n\to\infty$, semble al\'eatoire et certainement 
{\it non born\'ee sur l'ensemble des corps de nombres totalement r\'eels},
\`a $p$ constant. 

\smallskip
Par contre, \`a $k$ constant, nous avons conjectur\'e 
dans \cite[Conjecture 8.11]{Gra6} que $k$ est $p$-rationnel 
(i.e., ${\mathcal T}_k=1$) pour tout $p \gg 0$, en notant que, 
pour $p$ assez grand, ${\mathcal T}_k$ est r\'eduit au r\'egulateur 
${\mathcal R}_k$ (suite exacte \eqref{se1}) qui reste l'invariant crucial.

\smallskip
Comme d\'ej\`a dit, la non-$p$-rationalit\'e d'un corps de nombres $k$
(principalement totalement r\'eel) semble \^etre une 
obstruction irr\'eductible (\`a l'heure actuelle) \`a la preuve de 
nombreuses conjectures en th\'eorie de Galois sur $k$.
On peut penser que cela provient, quel que soit le cadre 
th\'eorique, de la nature des fonctions $L_p$ 
correspondantes, obtenues par interpolation de valeurs 
complexes, auquel cas, comme l'avait remarqu\'e Washington en 
1980/81 dans le cas de ${\mathcal T}_k$ (voir la bibliographie 
de \cite{Gra7}), la pr\'esence de \go Siegel zeroes\gf 
(i.e., tr\`es proches de $1$) rend ce type d'invariants cohomologiques 
probl\'ematiques, notamment lorsque $p$ varie ou tend vers l'infini.

\end{document}